\documentclass{amsart}

\usepackage{amssymb, latexsym, amsfonts, amsthm, amsmath}
\usepackage{enumerate, epsfig, subfigure, color, graphicx}

\numberwithin{equation}{section}

\theoremstyle{plain}
\newtheorem{theorem}{Theorem}

\newtheorem{lemma}[theorem]{Lemma}

\newtheorem{question}[theorem]{Question}

\theoremstyle{definition}

\newtheorem{remark}[theorem]{Remark}
\newtheorem*{remark*}{Remark}

\begin{document}
\title[sharp isoperimetric inequalities]{Sharp isoperimetric inequalities for infinite plane graphs with bounded vertex and face degrees}

\author[B. Oh]{Byung-Geun Oh}
\address{Department of Mathematics Education, Hanyang University, 222 Wangsimni-ro, Seongdong-gu, Seoul 04763, Korea}
\email{bgoh@hanyang.ac.kr}

\date{\today}

\begin{abstract}
We provide sharp bounds for the isoperimetric constants of infinite plane graphs (tessellations) with bounded vertex and face degrees. For example, if
 $G$ is a plane graph satisfying the inequalities $p_1 \leq \deg v \leq p_2$ for $v \in V(G)$ and $q_1 \leq \deg f \leq q_2$ for $f \in F(G)$, where
$p_1, p_2, q_1$, and $q_2$ are natural numbers  such that $1/p_i + 1/q_i  \leq 1/2$, $i=1,2$, then we show that 
\[
\Phi (p_1, q_1) \leq \inf_S \frac{|\partial S|}{|V(S)|} \leq \Phi (p_2, q_2), 
\]
where the infimum is taken over all finite nonempty subgraphs $S \subset G$, $\partial S$ is the set of edges connecting $S$ to $G \setminus S$, and $\Phi(p,q)$ is defined by
\[
\Phi (p, q) = (p-2) \sqrt{1 - \frac{4}{(p-2)(q-2)}}.
\]
For $p_1=3$ this gives an affirmative answer to a conjecture by Lawrencenko, Plummer, and Zha  from 2002, and for general $p_i$ and $q_i$
our result fully resolves a question posed in the book by Lyons and Peres from 2016, where they extended the conjecture of Lawrencenko et al.\ to the above form.
\end{abstract}

\maketitle

\section{Introduction}
Graphs considered in this paper are infinite tessellations of the plane, on which we study geometric and topological properties of the graph and prove sharp isoperimetric inequalities.
In particular, we provide an affirmative answer to Question~6.21 in the book  by Lyons and Peres \cite{LP16} as well as Conjecture~1.1 by  Lawrencenko, Plummer, and Zha  \cite{LPZ}.
The main tool  is the combinatorial Gauss-Bonnet theorem involving left turns (geodesic curvature), and we have obtained the results 
by interpreting left turns of the boundary curves combinatorially. The background and motivation of our research is the following.

Suppose $\mathcal{M}$ is a complete simply connected  open 2-dimensional  Riemannian manifold whose Gaussian curvature is bounded above by $k <0$. 
Then one can prove, using the famous Cartan-Hadamard theorem, that 
\begin{equation}\label{E:isoM}
\inf_\Omega \left\{ \frac{\mbox{length}(b \Omega)}{\mbox{area}(\Omega)} \right\} \geq  \sqrt{-k},
\end{equation}
where the infimum is taken over all finite regions  $\Omega \subset \mathcal{M}$ with smooth boundaries $b \Omega$  (cf.\ \cite[Theorem 34.2.6]{BZ}).
The constant (infimum) on the left hand side of \eqref{E:isoM} is  called the \emph{isoperimetric constant} of $\mathcal{M}$, 
or  the \emph{Cheeger constant}  named after Jeff Cheeger \cite{Che}, and is known to be equal to $\sqrt{-k}$ for the hyperbolic plane of constant curvature $k<0$.

It might be natural to expect a similar phenomenon for discrete cases. Let us choose a plane graph $G$ as a discrete analogue 
of a 2-dimensional  Riemannian manifold, and note that curvature on a plane graph is usually described in terms of vertex and face degrees. (See Section~\ref{prelim}
for notation and terminology,  and Section~\ref{S:CGB} for the concept and properties of combinatorial curvatures.)
Thus a counterpart to a simply connected 2-dimensional  Riemannian manifold of \emph{constant} Gaussian curvature is a $(p,q)$\emph{-regular graph} $G$,
a tessellation of the plane with $\deg v = p$ for all $v \in V = V(G)$ and $\deg f = q$ for all $f \in F = F(G)$, where $p$ and $q$ are natural numbers greater than or equal to $3$,
$V$ and $F$ are the vertex and face sets of $G$, respectively, and $\deg a$ denotes the degree of $a \in V\cup F$ (the number of edges incident to $a$).
In this setting Riemannian manifolds whose curvatures are bounded above would correspond to tessellations 
of the plane with vertex and face degrees bounded below, say by $p$ and $q$, respectively.

There are several options for a discrete analogue to the isoperimetric constant.  
Let $S$ be a subgraph of a given tessellation $G$, and we define the \emph{edge boundary} $\partial S$ of $S$ as  the set of edges in $E = E(G)$ with one end on 
$V(S)$  and the other end on $V \setminus V(S)$, where $E$ is the edge set of $G$. Also we define the \emph{boundary walk} $bS$ of $S$ 
as the sequence of  edges  traversed by those who walk along the topological boundary of $S$ in the positive direction. See Section~\ref{prelim} for the precise definition of $bS$,
but at this point one may think of it as the set of edges in $E(S)$ that are included in the boundaries of faces in $F \setminus F(S)$. 
Now the \emph{isoperimetric constants}, which are also known as the \emph{Cheeger constants} as in the continuous case, of $G$ are defined by the formulae
\begin{equation}\label{E:isoconst}
\begin{aligned}
& \imath (G) = \inf_S \frac{|\partial S|}{|V(S)|}, \qquad  && \imath^{*} (G) = \inf_S \frac{|bS|}{|F(S)|}, \\
& \imath_\sigma (G) = \inf_S \frac{|\partial S|}{\sum_{v \in V(S)} \deg v}, \qquad  && \imath_\sigma^{*} (G) = \inf_S \frac{|bS|}{\sum_{f \in F(S)} \deg f},
\end{aligned}
\end{equation}
where the infima are taken over all finite nonempty subgraphs $S \subset G$ 
and $|\cdot|$ is the cardinality of the given set. Here it is also assumed that $F(S) \ne \emptyset$ for $\imath^{*} (G)$ and $\imath_\sigma^{*} (G)$.

Isoperimetric constants have been studied extensively in graph theory, because they are related to many important properties of given structures and
therefore have numerous applications. See for instance \cite{BKW, Moh89} and the references therein. The constants $\imath(G)$ and $\imath_\sigma (G)$ are 
the most common isoperimetric constants, because they can be defined for every type of graphs including non-planar graphs, disconnected graphs, and
finite graphs (with some modification in the definition), and more importantly, because they have many applications in spectral theory on graphs, simple random walks, etc. For example,
$\imath(G)$  and $\imath_\sigma (G)$ are used to bound the bottom of the spectrum of the negative combinatorial Laplacian 
\cite{DK, Doz, DKarp,  Fuj,  Kel11, Kel17, KP11, KN, Moh88, Moh91}. These constants 
also appear in various other contexts \cite{ABH18, BKW,  BMS, Bol, HJL, HP20, Hig, HS, Kel10, KM, LevPer, LP16, Moh89,  Moh92, Moh02, Nic, Oh14, OS16, PRT, Woe98, Zuk} (and more).
For  $\imath^*(G)$ and $\imath^*_\sigma (G)$, they are defined by imitating the isoperimetric constants on Riemannian manifolds, 
hence these constants logically make sense only for graphs embedded into 2-dimensional manifolds. 
The constants  $\imath^*(G)$ or $\imath^*_\sigma (G)$ are found in \cite{Hig, HS, LPZ, Moh92, Oh15-2, OS16}, and positivity of these constants implies Gromov hyperbolicity 
of the graph, provided that face degrees are bounded \cite{Oh14, Oh15-2}. Note that $\imath^*(G)$ and $\imath^*_\sigma (G)$ are
the \emph{duals} of $\imath(G)$ and $\imath_\sigma(G)$, respectively; i.e., we have  $\imath (G) = \imath^* (G^*)$ and  $\imath_\sigma (G) = \imath_\sigma^* (G^*)$,
where $G^*$ is the dual graph of $G$.

Now let us go back to the problem considered at the beginning, and suppose that $G$ is a $(p,q)$-regular graph with $p, q \geq 3$.
If $1/p + 1/q > 1/2$, then $G$ is  in fact a finite graph and a tessellation of the Riemann sphere, which is not what we are studying in this paper. 
Anyway in this case $G$ becomes one of the platonic solids; i.e., G is one of the tetrahedron, the octahedron, the icosahedron, the cube, or the dodecahedron (Figure~\ref{F:platonic}).
\begin{figure}[t]
\centerline{
\hfill 
\subfigure{\epsfig{figure=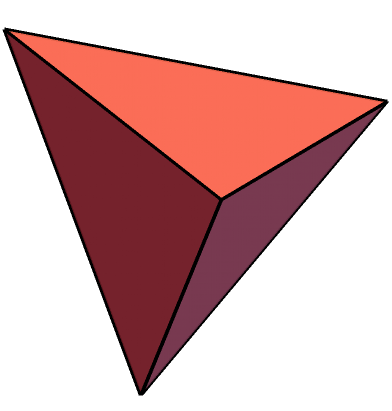, width=.7 in, height=.7 in}} \hfill
\subfigure{\epsfig{figure=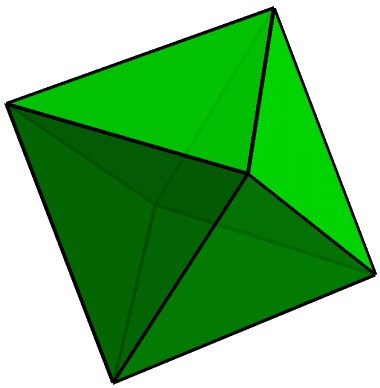, width=.7 in, height=.7 in}} \hfill
\subfigure{\epsfig{figure=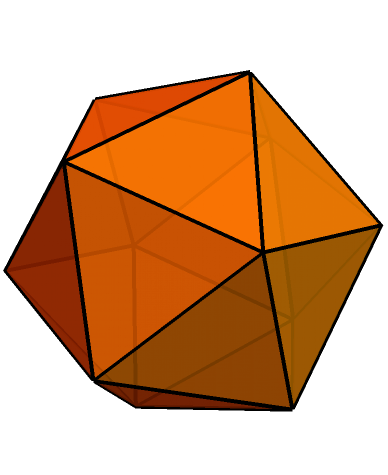, width=.7 in, height=.85 in}} \hfill
\subfigure{\epsfig{figure=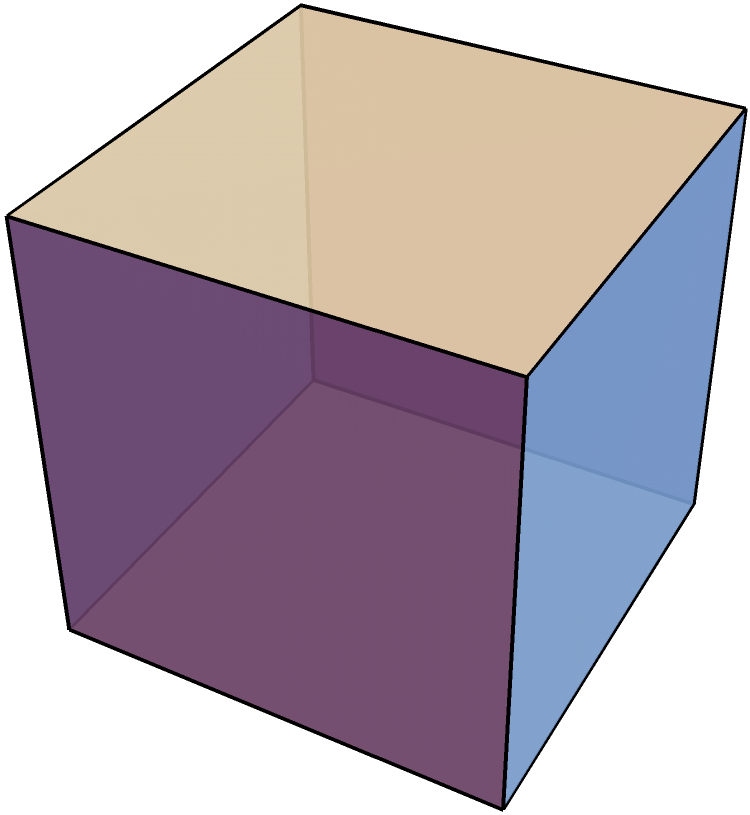, width=.7 in, height=.7 in}} \hfill 
\subfigure{\epsfig{figure=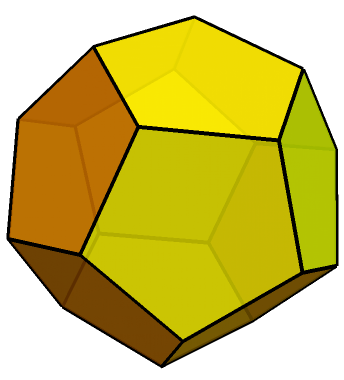, width=.67 in, height=.72 in}} \hfill}
\caption{Platonic solids}\label{F:platonic}
\end{figure}
The case $1/p + 1/q = 1/2$ occurs if and only if $(p,q) \in \{ (6,3), (4,4), (3,6) \}$, hence in this case $G$ is infinite and becomes one of the regular tilings of the plane; 
i.e., $G$ is the $6$-regular triangulation of the plane if $(p,q) = (6,3)$, the square lattice if $(p,q) =(4,4)$, and the hexagonal honeycomb if $(p,q) =(3,6)$ (Figure~\ref{F:grid}).
Graphs in this category correspond to simply connected 2-dimensional Riemannian manifolds of constant zero curvature (i.e., the Euclidean plane), 
and have zero isoperimetric constants in \eqref{E:isoconst}.
\begin{figure}[t]
\centerline{
\hfill 
\subfigure{\epsfig{figure=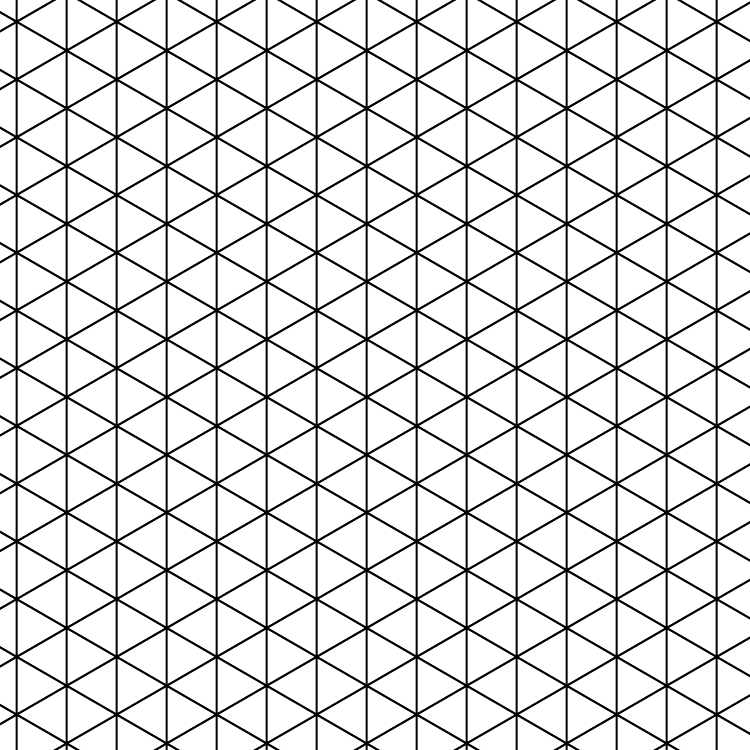, width=1.3 in}} \hfill
\subfigure{\epsfig{figure=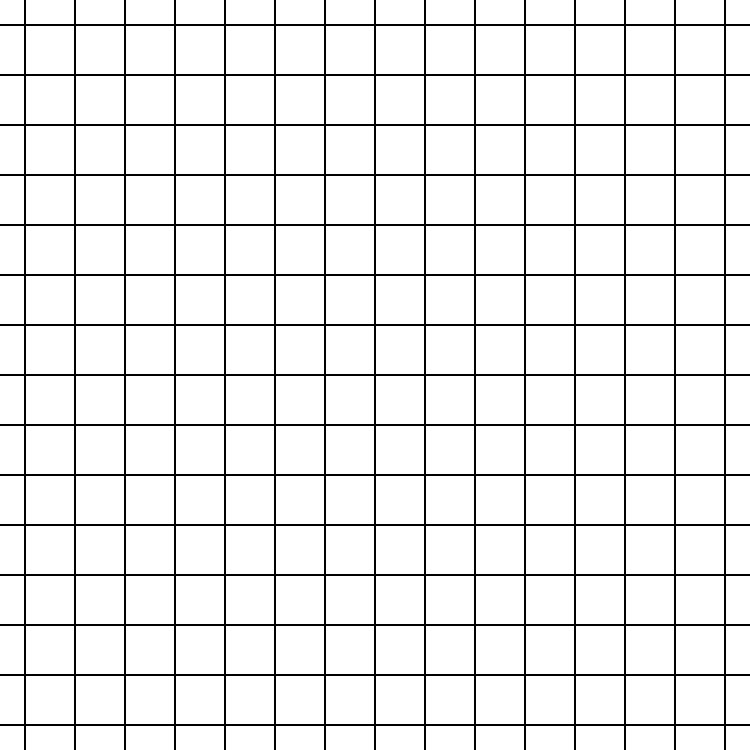, width=1.3 in}} \hfill
\subfigure{\epsfig{figure=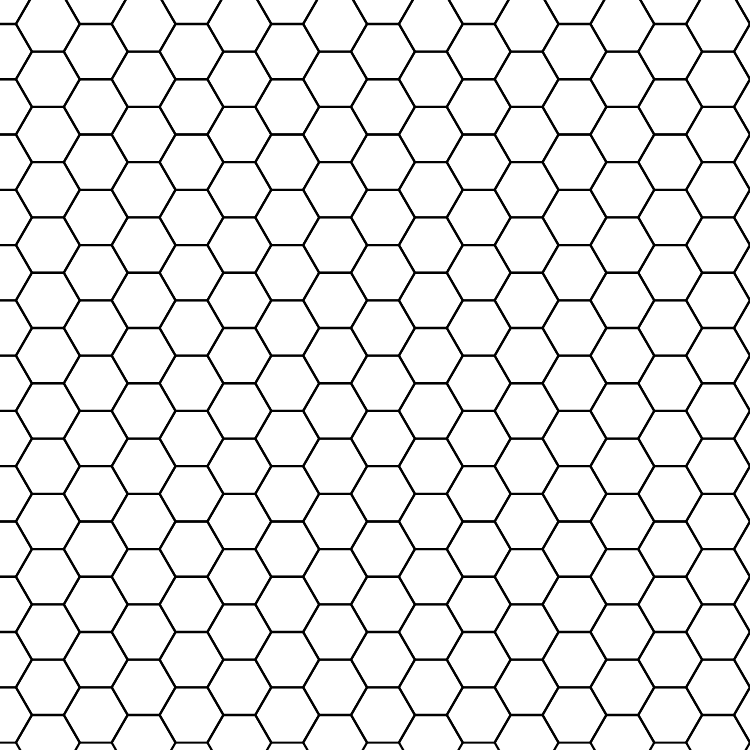, width=1.3 in}} \hfill}
\caption{Regular tilings of the plane}\label{F:grid}
\end{figure}

The case we are mostly interested in is when $1/p + 1/q < 1/2$; in this case the $(p, q)$-regular graph corresponds to a hyperbolic plane of constant negative curvature, and 
the isoperimetric constants in \eqref{E:isoconst} are positive \cite{Hig, Woe98, Zuk}. Moreover, the precise isoperimetric constants of $(p, q)$-regular graphs are computed independently by
two groups of researchers, H\"aggstr\"om,  Jonasson, and Lyons \cite{HJL} and Higuchi and Shirai \cite{HS}, and their results can be summarized as follows.

\begin{theorem}\label{T:PQreg}
Suppose $G$ is a $(p,q)$-regular graph for some integers $p, q \geq 3$ satisfying $1/p + 1/q \leq 1/2$. Then we have
\[
\begin{aligned}
\imath (G) & = \Phi (p, q)  \quad & \imath^{*} (G) & = \Phi (q, p) , \\
\imath_\sigma (G) & = \Phi (p, q)/p, \quad & \imath_\sigma^{*} (G) & = \Phi (q, p)/q,
\end{aligned}
\]
where
\[
\Phi (a, b)  = (a-2) \sqrt{1 - \frac{4}{(a-2)(b-2)}}.
\]
\end{theorem}

As we discussed above, a discrete counterpart to a Riemannian manifold whose curvature is bounded above by a negative number is a tessellation $G$ of the plane 
such that $\deg v \geq p$ for all $v \in V$ and $\deg f \geq q$ for all $f \in F$, where $p, q$ are natural numbers satisfying $1/p + 1/q < 1/2$. Then 
as in the continuous case \eqref{E:isoM}, one can expect that the isoperimetric constants of $G$ are bounded below by those of the $(p,q)$-regular graph; i.e., we can predict
for example that
\begin{equation}\label{C1}
\imath^* (G) \geq \Phi (q, p)  = (q-2) \sqrt{1 - \frac{4}{(p-2)(q-2)}},
\end{equation}
and similar expectation would apply to  the other isoperimetric constants. For this reason the inequality \eqref{C1} was conjectured by  Lawrencenko et al.\ \cite[Conjecture 1.1]{LPZ}
for the case $q=3$ (for $q > 3$ it was also conjectured in the same paper \cite{LPZ}, but only implicitly). 

Conversely if $G$ is an infinite tessellation of the plane and if $\deg v \leq p$ and $\deg f \leq q$ for all $v \in V$ and $f \in F$, where $p$ and $q$ are as before,
then it is also natural to expect that the isoperimetric constants of $G$ are not greater than those of the $(p,q)$-regular graphs. In this context 
Lyons and Peres posed the following question in \cite[Question~6.21]{LP16}.

\begin{question}\label{Q}
Suppose $G=(V, E, F)$ is a plane graph satisfying the inequalities $p_1 \leq \deg v \leq p_2$ for all $v \in V(G)$ and $q_1 \leq \deg f \leq q_2$ for all $f \in F(G)$, where
$p_1, p_2, q_1$, and $q_2$ are natural numbers  such that $1/p_i + 1/q_i  \leq 1/2$, $i=1,2$. Then does it always hold
\[
\Phi (p_1, q_1) \leq \imath(G) \leq \Phi (p_2, q_2) \, ?
\]
\end{question}

Before stating our results, let us briefly describe some known estimates for isoperimetric constants. 
The first known  
such estimate is due to Dodziuk \cite{Doz}, who proved that if $G$ is a triangulation (i.e., $\deg f =3$ for every face $f$) and $\deg v \geq 7$ for every $v \in V = V(G)$,
then we have
\[
\imath^*(G) \geq \frac{1}{26}.
\]
Dodziuk and Kendall  also proved in \cite{DK} that 
\[
\imath_\sigma (G) \geq \frac{1}{78}
\]
for  triangulations $G$ with $\deg v \geq 7$ for every $v \in V$. On the other hand, Dodziuk's bound for $\imath^*(G)$  was significantly improved by Mohar \cite{Moh92} such as
\[
\imath^* (G) \geq  \frac{p-6}{p-4},
\]
where $G$ is a triangulation of the plane with $\deg v \geq p \geq 7$ for all $v \in V$, and this result was further improved by Lawrencenko et al.\  \cite{LPZ} such as
\[
\imath^*(G) \geq \frac{(p-6)(p^2 - 8p +15)}{(p-4)(p^2 -8p +13)}.
\]
When $G$ is a tessellation of the plane 
with $\deg v \geq p$ for all $v \in V$ and $\deg f \geq q$ for all $f \in F$, where $p, q$ are natural numbers such that $1/p + 1/q < 1/2$, 
it was shown by Mohar \cite{Moh02} that
\[
\imath(G) \geq \frac{pq -2p -2q}{3q-8}.
\]
Also see \cite{KP11}, where the constants $\imath(G)$ and $\imath_\sigma(G)$ are estimated in terms of (vertex) combinatorial curvature. 
(See  Section~\ref{S:CGB} for the definition of combinatorial curvature.)

Now we present

\begin{theorem}\label{T1}
Suppose $G=(V,E,F)$ is a plane graph such that $p \leq \deg v < \infty$ and $q \leq \deg f < \infty$ for all $v \in V$ and $f \in F$, where $p$ and $q$
are natural numbers satisfying $1/p + 1/q \leq 1/2$. Then we have
\[
\begin{aligned}
\imath (G) & \geq \Phi (p, q)  \quad & \imath^{*} (G) & \geq \Phi (q, p) , \\
\imath_\sigma (G) & \geq \Phi (p, q)/p, \quad & \imath_\sigma^{*} (G) & \geq \Phi (q, p)/q.
\end{aligned}
\]
\end{theorem}

The conclusions of Theorem~\ref{T1} remain  true even when $G$ has some \emph{infinigons}; i.e., $G$ has some faces $f$ with $\deg f = \infty$. 
See Section~\ref{S:last} for details. We have excluded infinigons from the statement of Theorem~\ref{T1} 
just for simplicity, because our bounds for $\imath(G)$ and $\imath_\sigma (G)$ are obtained  through dual graphs.
Meanwhile, it is worth to mention that the graph $G$ in Theorem~\ref{T1} must be an infinite
tessellation of the plane, because its \emph{corner curvature} is always non positive. See \cite[Theorem~1]{Kel11}. 

We next provide our second result, which is about  upper bounds for isoperimetric constants.

\begin{theorem}\label{T2}
Suppose $G=(V,E,F)$ is an infinite tessellation of the plane such that $\deg v \leq p$ and $\deg f \leq q$ for all $v \in V$ and $f \in F$, where $p$ and $q$
are natural numbers satisfying $1/p + 1/q \leq 1/2$. Then we have
\[
\begin{aligned}
\imath (G) & \leq \Phi (p, q)  \quad & \imath^{*} (G) & \leq \Phi (q, p) , \\
\imath_\sigma (G) & \leq \Phi (p, q)/p, \quad & \imath_\sigma^{*} (G) & \leq \Phi (q, p)/q.
\end{aligned}
\]
\end{theorem}

Note that Theorems~\ref{T1} and \ref{T2} together completely resolve Question~\ref{Q} as well as the conjecture  by  Lawrencenko et al., and 
our results are the best possible because the isoperimetric
constants of $(p,q)$-regular graphs are the lower and upper bounds in Theorems~\ref{T1} and \ref{T2}, respectively. 

This paper is organized as follows. Preliminaries are given in Section~\ref{prelim}, where we introduce some unusual notation such as $D(S)$, $S^+$, and $S^-$.
The concept of combinatorial curvature and three versions of the combinatorial Gauss-Bonnet formula  involving boundary turns (geodesic curvature) 
are given in Section~\ref{S:CGB}. A lemma is proved in Section~\ref{S:Lemma}, and  using this lemma we
prove Theorem~\ref{T1}  in  Section~\ref{ProofT1}. Basically Sections~\ref{S:CGB}--\ref{ProofT1} are spent for the proof of Theorem~\ref{T1}.
Theorem~\ref{T2} is proved in Section~\ref{ProofT2}, where some stuffs in Sections \ref{S:CGB}--\ref{ProofT1} are used and some others are modified.
In Section~\ref{S:triangulation} we study another isoperimetric constant involving boundary vertices, and
the paper is finished in Section~\ref{S:last} where we discuss plane graphs with infinigons (called \emph{locally tessellating plane graphs}).

\section{Preliminaries}\label{prelim}
Suppose $G$ is a graph with the vertex set $V = V(G)$ and the edge set $E = E(G) \subset V \times V$. We say that
$G$ is \emph{connected} if it is connected as a one-dimensional simplicial complex,  \emph{infinite} if it has infinitely many vertices and edges, 
and \emph{simple} if it does not contain any multiple edges nor self loops; i.e., for $u, v \in V$ there is at most one edge $[u, v] \in E$ with endpoints $u$ and  $v$, and
every edge must have distinct endpoints. Because $G$ is always assumed to be an undirected graph in this paper, an edge of the form $[u, v]$ must be considered  the same as $[v, u]$.
If $[u, v] \in E$, the vertices $u, v$ are called \emph{neighbors} or \emph{adjacent}.

A graph $G$ is called \emph{planar} if there exists a continuous injective map $h : G \hookrightarrow \mathbb{R}^2$, 
and the embedded image $h(G)$ is called a \emph{plane graph}. In general a planar graph $G$ and its embedded plane graph $h(G)$ are different objects, 
but for simplicity we will not distinguish them and use the letter $G$ for its embedded graph $h(G)$.
Thus $G$ will be considered a subset of $\mathbb{R}^2$. Moreover, we will always assume that $G$ is embedded into $\mathbb{R}^2$ \emph{locally finitely}, which
means that every compact set in $\mathbb{R}^2$ intersects only finitely many vertices and edges of $G$.
The closure of each connected component of $\mathbb{R}^2 \setminus G$ is called a (closed) \emph{face} of $G$, and we denote by $F = F(G)$ the face set of $G$.
Because a plane graph $G$ is completely determined by its vertex, edge, and face sets, we will identify $G$ with the triple $(V, E, F)$ and use the notation $G = (V, E, F)$,
which is already used in the statements of Question~\ref{Q} and Theorems~\ref{T1} and \ref{T2}.

Note that in our definition each face is a closed set in $\mathbb{R}^2$. Similarly we will treat vertices and edges of $G$ as closed sets in $\mathbb{R}^2$, and
two objects in $V \cup E \cup F$ will be called \emph{incident} to each other if one is a proper subset of the other.
The \emph{degree} of a vertex $v \in V$ is the number of edges  incident to $v$, and similarly the \emph{degree} or 
\emph{girth} of a face $f \in F$ is the number of edges incident to $f$. The degrees of $v \in V$ and $f \in F$ will be denoted by $\deg v$ and $\deg f$, respectively.

Following \cite{BP01, BP06}, we call a connected simple plane graph a \emph{tessellation} or \emph{tiling} if the following conditions hold:
\begin{enumerate}[(a)]
\item every edge is incident to two distinct faces;
\item any two distinct faces are either disjoint or intersect in one vertex or one edge;
\item every face is a polygon with finitely many sides; i.e., if $f^\circ$ is a component of $\mathbb{R}^2 \setminus G$, then $f^\circ$ is homeomorphic
to the unit disk $\mathbb{D} \subset \mathbb{R}^2$, the topological boundary of $f^\circ$ is 
homeomorphic to a circle, and  we have $\deg f < \infty$, where $f = \overline{f^\circ} \in F$ is the closure of $f^\circ$.
\end{enumerate}
If $G$ is a tessellation, we should have  $\deg v = |E(v)|=|F(v)|$ and $\deg f = |V(f)|= |E(f)|$, where  $|\cdot|$ denotes the cardinality of the given set. Here
$V(a)$, $E(a)$, and $F(a)$  denote the sets of vertices, edges, and faces, respectively, incident to $a \in V \cup E \cup F$. 
Note that we must have $\deg v < \infty$ for every $v \in V$, because we have assumed that $G$ is embedded into $\mathbb{R}^2$ \emph{locally finitely}.
 
Following \cite{KP11}, we call $G$ a \emph{locally tessellating} plane graph, or a \emph{local tessellation} of the plane, 
if $G$ is a connected simple plane graph satisfying (a) and (b) above, and  (c$'$) below instead of (c):
\begin{enumerate}[(c$'$)]
\item every face is a polygon with finitely or infinitely many sides; i.e., a component $f^\circ$ of $\mathbb{R}^2 \setminus G$ either satisfies the statement (c) above, or
it is homeomorphic to the upper half plane $\mathbb{R} \times \mathbb{R}^+ = \{ (x, y) \in \mathbb{R}^2 : y > 0 \}$, the topological boundary of $f^\circ$ is
homeomorphic to the $x$-axis, and $\deg f = \infty$ with  $f = \overline{f^\circ}$.
\end{enumerate}
That is, a local tessellation is almost the same as tessellations except that it may have faces of infinite degrees, called \emph{infinigons}. 
Examples of local tessellations include tessellations of the plane, and infinite trees such that $\deg v \geq 3$ for every vertex $v$. 
However, we will not consider infinigons until Section~\ref{S:last}, the last section of the paper, so one can mostly regard $G$ as just a tessellation.
 
For $V(S) \subset V$, $E(S) \subset E$, and $F(S) \subset F$, the triple $S = (V(S), E(S), F(S))$ is a  \emph{subgraph} of $G=(V, E, F)$ if $V(e) \subset V(S)$ for every $e \in E(S)$
and $E(f) \subset E(S)$ for every $f \in F(S)$. In this case we use the notation $S \subset G$.
Remark that our definition for the face set $F(S)$ of $S$ is different from the usual definition, because in our definition $F(S)$ has to be a subset of $F$ and a nonempty 
subgraph might have the empty face set. A subgraph $S \subset G$ is called \emph{induced} if it is induced by its vertex set; i.e.,
$S$ is induced if for $e \in E$ we have $e \in E(S)$ whenever $V(e) \subset V(S)$, and for $f \in F$ we have  $f \in F(S)$ whenever $V(f) \subset V(S)$. 
We call $S \subset G$ a \emph{face graph} if $V(S) = \bigcup_{f \in F(S)} V(f)$ and $E(S) = \bigcup_{f \in F(S)} E(f)$. Definitely face graphs are the subgraphs
consisting of faces. If $S$ is a connected subgraph and $\chi(S) = |V(S)| - |E(S)| + |F(S)|= 1$, where $\chi(\cdot)$ denotes the Euler characteristic of $S$, then $S$ will be called
\emph{simply connected}. A \emph{polygon} is a simply connected face graph. Note that a face $f \in F$ itself can be considered a subgraph classified as a polygon.

A \emph{path} is a sequence of vertices of the form $[v_0, v_1, \ldots, v_n]$ such that $[v_{i-1}, v_{i}] \in E$ for all $i = 1,2, \ldots, n$. In this case we will say that the \emph{length}
of the path is $n$, and the path connects (or joins) the initial vertex $v_0$ to the terminal vertex $v_n$. A path is called \emph{cycle} if its initial and terminal vertices coincide,  
and \emph{simple} if no vertices appear more than once except the case that the path is a cycle
and the only repetition is the initial and terminal vertices. We remark that there is a path of \emph{zero} length; i.e., a path could be of the form $[v]$ for some $v \in V$.
However, because we consider only simple graphs which do not have self loops, there is no cycle of length one; i.e., the length of a cycle must be either zero or at least two.
A \emph{walk} will mean a  union of paths, which usually arises when one walks along the boundary of some region.
We will regard a path $\gamma=[v_0, v_1, \ldots, v_n]$ as a subgraph of $G$ with $V(\gamma) = \{ v_i : i =0,1, \ldots, n\}$,
$E(\gamma)  = \{ [v_{i-1}, v_i] : i =1,2, \ldots, n\}$, and $F(\gamma) = \emptyset$. Similar definition will be applied to walks. However, 
paths (or walks) have one more structure than usual subgraphs: the \emph{orientation}. 

Our definition for paths or cycles is different from the usual one in graph theory, because we allow some repetitions of vertices or edges in a path.
See \cite{BoMur, Dies} for the usual definition for paths, cycles, and walks in graph theory. 
If we need to mention paths (or cycles) without repetitions of vertices, we will use the terminology \emph{simple path} (or cycle, respectively) as defined in the previous paragraph.

Suppose a nonempty subgraph $S \subset G$ is given. Then we define the region $D(S)$ determined by $S$ as
\[
D(S) = \left( \bigcup_{f \in F(S)} f \right) \cup \left( \bigcup_{e \in E(S)} e \right) \cup \left( \bigcup_{v \in V(S)} v \right).
\]
The cycle or the union of cycles obtained by walking along the topological boundary of $D(S)$ in the positive orientation is called the \emph{boundary walk} of $S$ and denoted by $b S$.
Note that if $S$ is connected and $\chi(S) = 2 -m$, that is, if $\mathbb{R}^2 \setminus D(S)$ has $m$ components, then $b S$  can be written as a union of $m$ 
cycles.  In other words, in this case we can write $b S = \Gamma_1 \cup \Gamma_2 \cup \cdots \cup \Gamma_m$,
where each $\Gamma_j$ is a cycle corresponding to the boundary of a component of $\mathbb{R}^2 \setminus D(S)$. Now suppose that 
$S$ is not connected (but finite); i.e., let us assume that $S= S_1 \cup \cdots \cup S_k$, where each $S_i$ is a connected component of $S$. 
Then each $b S_i$ can be written as a union of cycles as above, hence so can be $b S = b S_1 \cup \cdots \cup b S_k$. 
Now we define $|b S|$ as the sum of lengths of the cycles consisting of $b S$. 
Remark that in this paper the notation $| \cdot |$ denotes mostly the cardinality of the given set, but
it will mean the length if the given object is a path or a walk, as in the case $| bS |$. 
Next we define the edge boundary $\partial S$ of $S$, which was in fact already defined in the introduction, as the set of edges
connecting $V(S)$ to $V \setminus V(S)$. 

For $v, w \in V$,  the combinatorial distance $d(v, w)$ between $v$ and $w$ is the minimum of the lengths of paths joining $v$ and $w$. For $v_0 \in V$ and $n \in \mathbb{N}$,
the combinatorial ball $B_n (v_0)$ of radius $n$ and centered at $v_0$ is the induced subgraph whose vertex set consists of the vertices
$v \in V$ satisfying $d (v, v_0) \leq n$. Now for a given nonempty 
subgraph $A \subset G$, let $A^+$ be the \emph{face} graph consisting of all the faces incident to the vertices in $V(A)$. We let $\mathcal{B}_1 = A^+$, and the \emph{quasi-balls} $\mathcal{B}_n := \mathcal{B}_n(A)$ 
with \emph{height} $n \in \mathbb{N}$ and \emph{core} $A$ are defined inductively by $\mathcal{B}_{k+1} = \mathcal{B}_k^+$ for all $k \in \mathbb{N}$.
Conversely, for $S \subset G$ we define $S^-$ as the \emph{induced} subgraph of $S$ whose vertex set is $V(S) \setminus V(bS)$. That is, $S^-$ is the subgraph of $S$ 
whose vertices lie in the interior of $D(S)$, and it has to be induced. Note that even though $S$ does not have to be induced, $S^-$ is always induced by the definition.
Finally we define the \emph{depth} of a subgraph $S$ as follows: $S$ is of depth $0$ if $S^-$ is empty, and inductively we call $S$ is of 
depth $n$ if $S^-$ is of depth $n-1$. 

\section{Combinatorial Gauss-Bonnet Theorem}\label{S:CGB}
Let $G = (V, E, F)$ be a tessellation of the plane. For each $v \in V$, we define the (vertex) \emph{combinatorial curvature} $\kappa (v)$ at $v$ by 
\begin{equation}\label{E:Comv}
\kappa(v) = 1 - \frac{\deg v}{2} + \sum_{f \in F(v)} \frac{1}{\deg f}.
\end{equation}
If $V_0$ is a finite subset of $V$,  we define $\kappa(V_0) = \sum_{v \in V_0} \kappa (v)$, and for a finite subgraph $S \subset G$ we use the notation 
$\kappa (S) := \kappa (V(S))$.

It is not clear when the concept of combinatorial curvature arose, but it was already considered by Descarte for convex polyhedra (cf.\,\cite{Fed82}), and studied
by Nevanlinna in the early 20th century \cite{Nev}. The current formula \eqref{E:Comv} is due to Stone \cite{Sto76}, and the idea was used in \cite{Gro} as well. 
Since then properties of combinatorial curvature have been extensively studied by many researchers 
\cite{AH23, BP01, BP06, Chen09, CC08, DeMo07, Ghi23, Hig, HJ15, HuS19, HS20, HS22, Kel10, Kel11, Kel17, KP11, KPP,  Oh14, Oh17, Oh22, OS16, Old17, SY, Woe98, Zuk}
(and more).

The meaning of the combinatorial curvature is as follows. We associate each face $f \in F$ with a regular ($\deg f$)-gon of side lengths one, and we paste these
regular polygons along sides by the way that the corresponding faces are pasted. Then the resulting surface, which we denote by $\mathcal{S}_G$, becomes 
a metric surface called a \emph{polyhedral surface},
a special type of Aleksandrov surfaces \cite{AZ, Res}. Definitely $\mathcal{S}_G$ is homeomorphic to the Euclidean plane, hence we can assume that $\mathcal{S}_G$
includes $G$ in a natural way. Also it is not difficult to see that $\mathcal{S}_G$ is locally isometric to subsets of the Euclidean plane except at the vertices of $G$, 
and at each $v \in V \subset \mathcal{S}_G$  the \emph{total angle} $T(v)$ becomes $\sum_{f \in F(v)} (\pi - 2 \pi/\deg f)$. Therefore the atomic curvature at $v$ is 
\begin{equation}\label{E:IC}
\omega (\{ v \} ) = 2 \pi - T(v) = 2 \pi - \sum_{f \in F(v)} \left( \pi - \frac{2 \pi}{\deg f} \right) = 2 \pi \cdot \kappa (v).
\end{equation}
Here $\omega (\cdot )$ denotes the integral curvature defined on Borel sets of $\mathcal{S}_G$. Thus the combinatorial curvature $\kappa$ is nothing but the usual 
integral curvature defined on the polyhedral surface $\mathcal{S}_G$, but it is normalized so that $2 \pi$ corresponds to $1$ because we do not want to carry $2 \pi$
in every formula. 

In differential geometry perhaps the Gauss-Bonnet formula is one of the most basic and useful tools related to curvature. It seems not much different in the discrete setting either,
but there are several versions for the \emph{combinatorial} Gauss-Bonnet formula in graph theory. Let $G$ be a tessellation embedded locally finitely into a 2-dimensional compact 
manifold $\mathcal{M}$. Then $G$ must be a finite graph, and one can show that
\begin{equation}\label{E:BGB}
\kappa(G) = \sum_{v \in V} \kappa(v) = \chi (\mathcal{M}),
\end{equation}
where $\chi (\mathcal{M})$ is the Euler characteristic of $\mathcal{M}$ (cf.\ \cite{BP01, Chen09, DeMo07, Oh17}).
We will call \eqref{E:BGB} the basic form of the Gauss-Bonnet formula.
The Gauss-Bonnet formulae we need, however, are more complicated than \eqref{E:BGB}, and we have to introduce some more notation.

Let $\Gamma = [v_0, v_1, \cdots, v_n = v_0]$ be a cycle in a tessellation $G$ of the plane. Assume that $n \geq 2$, 
and for  $k \in \{ 1,2, \ldots, n\}$ let $f_1, f_2, \ldots, f_s$ be the faces in $F$
that are incident to $v_k$ and lies on the \emph{right} of $[v_{k-1}, v_k, v_{k+1}]$, where we interpret $v_{n+1} = v_1$ if $k = n$. 
Then the \emph{outer left turn} occurred  near $v_k$ is defined by 
\begin{equation}\label{outerLT}
\tau_o (v_k; v_{k-1}, v_{k+1}) =   \sum_{j=1}^{s} \left( \frac{1}{2} - \frac{1}{\deg f_j} \right) - \frac{1}{2},
\end{equation}
and the outer left turn of $\Gamma$ is defined by the formula
\[
\tau_o (\Gamma) = \sum_{k=1}^n \tau_o (v_k; v_{k-1}, v_{k+1}). 
\]
\begin{figure}[t]
 \centering
 \subfigure[the case $v_{k-1} \ne v_{k+1}$]{\scalebox{.9}{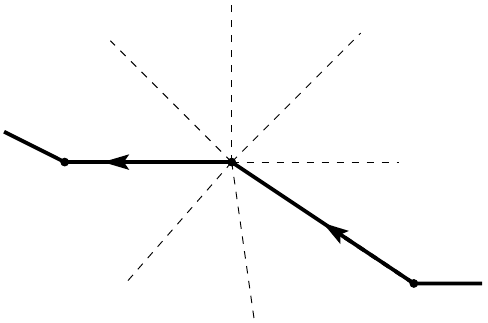}}  \\
 \subfigure[the case $v_{k-1} = v_{k+1}$ for $\tau_o$]{\scalebox{.9}{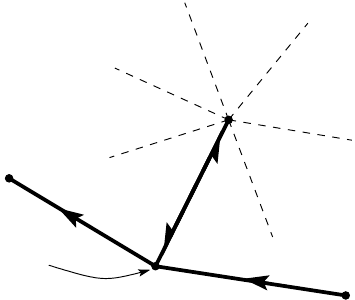}} \hspace{.1 in}
  \subfigure[the case $v_{k-1} = v_{k+1}$ for $\tau_i$]{\scalebox{.9}{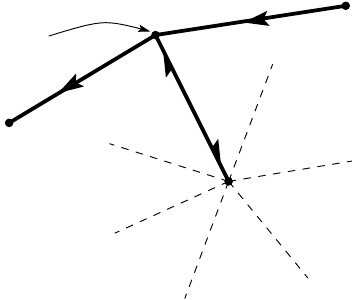}} 
\caption{Faces incident to $v_k$: the faces $f_j$ are on the right of $[v_{k-1}, v_k, v_{k+1}]$, and the faces $g_j$ are on the left of $[v_{k-1}, v_k, v_{k+1}]$.  
The case (b) will not appear for inner left turns, and the case (c) will not appear for outer left turns.}\label{F:leftturn}
\end{figure}
Similarly, let $g_1, g_2, \ldots, g_t$ be the faces in $F$
that are incident to $v_k$ and lies on the \emph{left} of $[v_{k-1}, v_k, v_{k+1}]$. Then the \emph{inner left turn} occurred near $v_k$ is defined by 
\begin{equation}\label{innerLT}
\tau_{i} (v_k; v_{k-1}, v_{k+1}) =   \frac{1}{2} - \sum_{j=1}^{t} \left( \frac{1}{2} - \frac{1}{\deg g_j} \right),
\end{equation}
and the inner left turn of $\Gamma$ is defined by the sum
\[
\tau_i (\Gamma) = \sum_{k=1}^n \tau_i (v_k; v_{k-1}, v_{k+1}).
\]
Here we remark that if $v_{k-1} = v_{k+1}$, then all the faces incident to $v_k$ should be considered lying both on the right and on the left of $[v_{k-1}, v_k, v_{k+1}]$,
hence in this case we must  have $s = t = \deg v_k$  and 
\[
\kappa (v_k) = \frac{1}{2}  -\tau_{o} (v_k;  v_{k-1}, v_{k+1}) = \frac{1}{2} + \tau_{i} (v_k;  v_{k-1}, v_{k+1}) .
\]
See Figure~\ref{F:leftturn}.
Finally if $\Gamma$ is a cycle of zero length, that is, if $\Gamma = [v_0]$ for some $v_0 \in V$,  we define $\tau_{o} (\Gamma) = 1 - \kappa(v_0)$ and
 $\tau_i (\Gamma) = \kappa(v_0) -1$. Note that the quantity $2 \pi  (1 - \kappa(v_0))$ is the total angle $T(v_0)$ at the point $v_0 \in \mathcal{S}_G$.

To explain the meaning of the inner and outer left turns defined above, suppose a cycle $\Gamma$ is given. 
Though we will consider more complicated cases later, here let us assume for simplicity that
$\Gamma$ is a simple cycle enclosing a polygon, say $S$, in the positive direction. Moreover, we regard $\Gamma \subset G$ as a set lying in the polyhedral surface  $\mathcal{S}_G$.
Then we imagine that  a person stands  one step to the right  from $\Gamma$,  and walks side by side along $\Gamma$.
Then  $2 \pi \cdot \tau_{o} (v_k; v_{k-1}, v_{k+1})$ would be the angle by which he or she turns to the left near $v_k$ (in the surface $\mathcal{S}_G$), and 
the total left turn made after a complete rotation along $\Gamma$ would be $2 \pi \cdot \tau_{o} (\Gamma)$. Note that in this case all the vertices of $S$ 
will be inside the path traversed by that person.
For the inner left turn, we think that this person stands one step to the left from $\Gamma$
and walks, and observe that $2 \pi \cdot \tau_{i} (v_k; v_{k-1}, v_{k+1})$ is the left turn made near $v_k$.  
Thus $2 \pi \cdot \tau_{i} (\Gamma)$ will become the total left turn made after a complete rotation along $\Gamma$, and in this case only the vertices in $S^-$
will be inside the path along which the person traveled. See Figure~\ref{F:leftturn-meaning}.
\begin{figure}[t]
\begin{center}
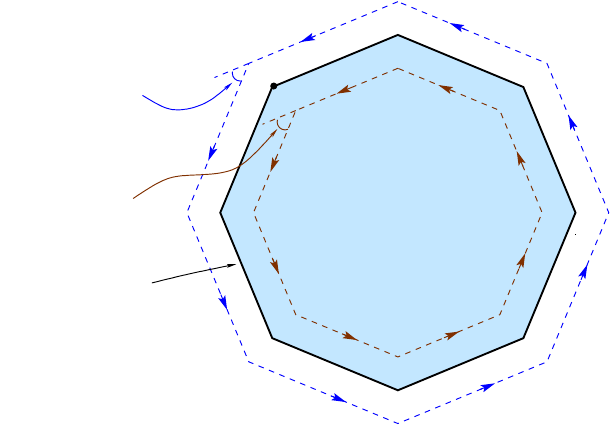
\caption{Meaning of outer and inner left turns.}\label{F:leftturn-meaning}
\end{center}
\end{figure}
Also note that, because $\Gamma$ was assumed to be simple, we have 
\begin{align*}
\tau_i (\Gamma) - \tau_o (\Gamma) & = \sum_{k=1}^n \bigl(\tau_{i} (v_k; v_{k-1}, v_{k+1}) - \tau_{o} (v_k; v_{k-1}, v_{k+1}) \bigr)\\
& =\sum_{k=1}^n \kappa(v_k) = \kappa (\Gamma).
\end{align*}

As we explained at the end of the previous section, if $S$ is a connected finite subgraph of $G$ then we can write 
$b S = \Gamma_1 \cup \Gamma_2 \cup \cdots \cup \Gamma_m$,
where each $\Gamma_j$ is a cycle corresponding to the topological boundary of a component of  $\mathbb{R}^2 \setminus D(S)$. In this case we define $\tau_o (bS)$ as
$\tau_o (bS) = \sum_{j=1}^m \tau_o (\Gamma_j)$. If $S$ is disconnected, that is, if $S = S_1 \cup \cdots \cup S_k$ with connected components $S_j$'s,
then we define $\tau_o (b S) = \sum_{j=1}^k \tau_o (b S_j)$. Now we are ready to describe our first Gauss-Bonnet formula. 

\begin{theorem}[Combinatorial Gauss-Bonnet Theorem-Type I]\label{CGBT1}
Suppose $G$ is an infinite tessellation and $S \subset G$ a finite subgraph of $G$, which is not necessarily connected. Then we have
\begin{equation}\label{GBF-1}
\kappa(S) + \tau_{o} (b S) = \chi (S),
\end{equation}
where  $\chi (S) = |V(S)| - |E(S)| + |F(S)|$, the Euler characteristic of $S$.
\end{theorem}
\begin{proof}
We regard $D(S)$ as a set lying in the polyhedral surface $\mathcal{S}_G$, and for sufficiently small $\epsilon>0$ let $P$ be the region in $\mathcal{S}_G$
which is obtained from the closed $\epsilon$-neighborhood of $D(S)$ by \emph{sharpening the corner} (see Figure~\ref{F:sharpen}). 
\begin{figure}[t]
\begin{center}
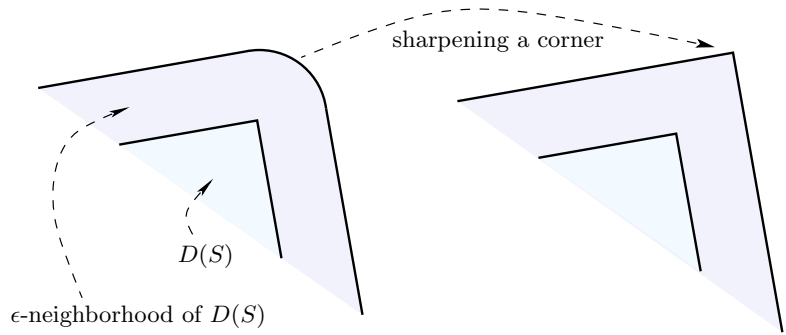
\caption{Obtaining the region $P$ by sharpening corners. This operation does not change geodesic curvature of the boundary, but it makes it easier to compare
the geodesic curvature of $b P$ with $\tau_o (b S)$.}\label{F:sharpen}
\end{center}
\end{figure}
That is, we slightly expand $D(S)$  in order to obtain $P$, and observe that the topological boundary $b P$ of $P$ becomes the union of paths traversed by a person 
walking in a position one step away to the right from $b S$. Let $P^\circ$ be the topological interior of $P$. Then the total left turn $\tau (b P)$, or the total geodesic curvature, 
of $bP$ is nothing but $2 \pi \cdot \tau_{o} (\Gamma)$ as we observed before. 
Moreover, the integral curvature $\omega (P^\circ)$ of $P^\circ$ is the same as $2 \pi \cdot \sum_{v \in V(S)} \kappa (v) = 2 \pi \cdot\kappa (S)$ by \eqref{E:IC}, 
because $\mathcal{S}_G$ is locally Euclidean except at the vertices of $G$. Finally it is clear that the Euler characteristic $\chi(P)$ of $P$ 
is the same as that of $S$. Thus Theorem~\ref{CGBT1} follows from the Gauss-Bonnet Theorem for polyhedral surface \cite[p.\,214]{AZ}, which says that
\begin{equation}\label{GBT-S}
\omega (P^\circ) + \tau (b P) = 2 \pi \cdot \chi(P).
\end{equation}
This completes the proof of Theorem~\ref{CGBT1}.
\end{proof}

\begin{remark}\label{R-face}
The equation \eqref{GBF-1} is an intrinsic property of $S$ itself and does not depend on the environment surrounding $S$. That is, the quantity $\deg f$ of a face $f \notin F(S)$
has nothing to do with whether \eqref{GBF-1} holds true, hence one can use any number for $\deg f$ as long as the same number is applied to both $\kappa(S)$ and $\tau_o (bS)$.
To prove this claim, suppose we write $b S = \Gamma_1 \cup \Gamma_2 \cup \cdots \cup \Gamma_m$ as explained in the previous section and in Theorem~\ref{CGBT1}.
Let a face $f \in F \setminus F(S)$ be given. Fix a vertex $v \in V(bS) \cap V(f)$ if any, and note that among the cycles $\Gamma_1, \Gamma_2, \ldots, \Gamma_m$ 
there exists exactly one $\Gamma_j$ that passes through $v$ where $f$ lies on the right of the cycle. This is because 
each cycle $\Gamma_j$ corresponds to a component of $\mathbb{R}^2 \setminus D(S)$, and among the components only one includes the interior of $f$.
Therefore with respect to the vertex $v$, the quantity $\deg f$ appears exactly twice in \eqref{GBF-1}: one in  $\kappa(v)$ and the other in $\tau_o (bS)$. 
But we have  $+1/ \deg f$ in the computation of $\kappa(v)$ while we have $- 1/ \deg f$ in that of $\tau_o(bS)$, 
so they should be canceled out. Because the values related to $f$ can contribute to $\kappa(S)$ or $\tau_o (bS)$  only  through  vertices in $V(bS) \cap V(f)$ and the above argument holds for
every vertices in $V(bS) \cap V(f)$, we conclude that the claim is true; i.e., in \eqref{GBF-1} we can replace  $\deg f$ by any value we want 
as long as the values we use  are the same both in $\kappa(v)$ and in $\tau_o (bS)$.
\end{remark}

To deduce the second combinatorial Gauss-Bonnet formula, suppose a finite subgraph $S \subset G$ is given. 
Furthermore, let us assume for a moment that $S$ is a \emph{face graph}, and suppose that  the interior
of $D(S)$ has $m$ connected components, say $D_1, D_2, \ldots, D_m$.  For each $j =1,2, \ldots, m$, let  $S_j$ be the face  subgraph of $S$ such that 
$D(S_j) = \overline{D}_j$, where $\overline{(\cdot)}$ is the topological closure of the given set. 
Then we can write $b S_j = \gamma_1^j \cup \gamma_2^j \cup \cdots \cup \gamma_{l_j}^j$ if $\mathbb{R}^2 \setminus  D_j$
has $l_j$ components, where each $\gamma_k^j$ is a cycle corresponding to the boundary of a component of $\mathbb{R}^2 \setminus  D_j$.  We remark 
at this point that what we consider are the components of $\mathbb{R}^2 \setminus  D_j$, not the components of $\mathbb{R}^2 \setminus  \overline{D}_j$ 
(see $\gamma_2$ in Figure~\ref{F:bS}). It is also worth to mention that no $\gamma_j$ is of length zero. Thus for a face graph $S$ we can write 
\begin{equation}\label{E:innerwalk}
b_i S := b S = bS_1 \cup b S_2 \cup \cdots \cup b S_m = \gamma_1 \cup \cdots \cup \gamma_l,
\end{equation}
 where $l=l_1 + l_2 + \cdots + l_m$ and each $\gamma_k$ corresponds to
the boundary of a component of complements of a component of $D(S)^\circ$. Now if $S$ is not a face graph,  we remove from $S$ all the edges and vertices of 
$S$ that are not incident to faces in $F(S)$, and we obtain a face graph $S_0$. Then we define the \emph{inner boundary walk} of $S$ by
$b_i S = b_i S_0$ using \eqref{E:innerwalk}. See Figure~\ref{F:bS} for the difference between the usual boundary walk $bS$ and the inner boundary walk $b_i S$.
In fact, in the proof of Theorem~\ref{CGBT2} we will obtain a new region by \emph{shrinking} $D(S)$, then
$b_i S$ will be the walk that is topologically equivalent to the boundary of this new region.
\begin{figure}[t]
\begin{center}
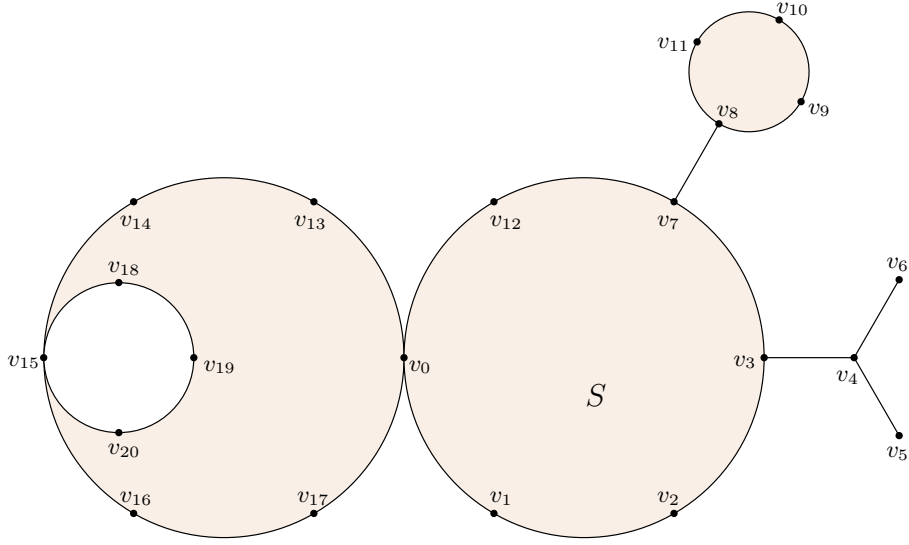
\caption{A subgraph $S$ with 21 vertices $v_0, v_1, \ldots, v_{20}$ on the boundary $bS$. For this graph we can write $bS=\Gamma_1 \cup \Gamma_2$ with
$\Gamma_1 = [ v_0, v_1, \ldots, v_5, v_4, v_6, v_4, v_3, v_7, \ldots, v_{11}, v_8, v_7, v_{12}, v_{0}, v_{13}, \ldots, v_{17},v_0]$ and 
$\Gamma_2 = [v_{15}, v_{18}, v_{19}, v_{20}, v_{15}]$. For the inner boundary walk $b_i S$ we have 
$b_i S = \gamma_1 \cup \gamma_2 \cup \gamma_3$ with $\gamma_1 = [v_0,  \ldots, v_3, v_7, v_{12}, v_0]$, 
$\gamma_2 = [v_0, v_{13}, v_{14}, v_{15}, v_{18}, v_{19}, v_{20}, v_{15}, v_{16}, v_{17}, v_0]$, and $\gamma_3 = [v_8, \ldots, v_{11}, v_8]$.}\label{F:bS}
\end{center}
\end{figure}

Now we are ready to present the second version of the Gauss-Bonnet formula we need. Note that if $b_i S$ is written as in \eqref{E:innerwalk},
we define $\tau_i (b_i S) = \sum_{j=1}^l \tau_i (\gamma_j)$.

\begin{theorem}[Combinatorial Gauss-Bonnet Theorem-Type II]\label{CGBT2}
Suppose $G$ is an infinite tessellation and $S \subset G$ a finite subgraph of $G$. Then we have
\begin{equation}\label{GBF-2}
\kappa(S^{-}) + \tau_{i} (b_i S) = \chi (D(S)^\circ),
\end{equation}
where $\chi (D(S)^\circ)$ denotes the Euler characteristic of $D(S)^\circ$.
\end{theorem}
Recall that we have $\chi (D(S)^\circ) = 2 - l$ when $D(S)^\circ$ is connected and $\mathbb{R}^2 \setminus D(S)^\circ$ has $l$ connected components. 
For the case that $D(S)^\circ$ is disconnected, we have $\chi (D(S)^\circ) = \chi (D_1) + \cdots + \chi (D_m)$, where  $D_1, \ldots, D_m$ are the connected components of $D(S)^\circ$.
For example we have $\chi (D(S)^\circ) = 1+ 1+ 1 =3$ for the subgraph $S$ in Figure~\ref{F:bS}, although we have $\chi(S) = 2-2 =0$. Also note that $\chi(S^-)$ is
in general different from $\chi (D(S)^\circ)$ (cf.\ Figure~\ref{F:inoutedge}).

\begin{proof}[Proof of Theorem~\ref{CGBT2}]
As in the proof of Theorem~\ref{CGBT1} we regard $D(S)$ as a set in $\mathcal{S}_G$, and let $(bS)_\epsilon$ be the (open) $\epsilon$-neighborhood of $bS$
for sufficiently small $\epsilon$. Let $P$ be the region in $\mathcal{S}_G$ which is obtained by sharpening corners of $D(S) \setminus (bS)_\epsilon$.
Then definitely $\chi (P) = \chi (D(S)^\circ)$, $\omega (P^\circ) = 2 \pi \cdot \kappa(S^-)$, and $\tau (bP) = 2 \pi  \cdot \tau_{i} (b_i S)$.
Therefore \eqref{GBF-2} comes from \eqref{GBT-S} as in the proof of Theorem~\ref{CGBT1}.
\end{proof}

We need one more combinatorial Gauss-Bonnet formula for the proof of Theorem~\ref{T2}. 
Suppose $S$ is a \emph{face} subgraph of $G$ such that $D(S)^\circ$ is connected, and let $T$ be the induced subgraph of $G$ with
$V(T) = V \setminus V(S)$. Then $T$ must be infinite, but it is not difficult to see that $bT$ is finite.
Note that $T$ could be disconnected, but every component of $D(T)$ is simply connected because $D(S)^\circ$ is connected.
Therefore we can write 
\begin{equation} \label{E:gamma_j}
b T = \gamma_1 + \gamma_2 + \cdots + \gamma_m,
\end{equation}
 where $m$ is the number of the components of $D(T)$ and each $\gamma_j$ is a cycle corresponding to the topological boundary of a component of $D(T)$. 
Now we recall that $S^+$ is the face graph consisting of the faces incident to the vertices of $S$, and define another boundary walk $b_1 S^+$ of $S^+$ by
$$b_1 S^+ := -b T,$$ 
where the negative sign represents the opposite orientation. 

The reason why we introduce $b_1 S^+$  is because we want an \emph{inner boundary walk} of $S^+$  passing through \emph{all} the vertices in $V(S^+) \setminus V(S)$
and enclosing exactly  the vertices in $V(S)$. The boundary walk $b_i S^+$ fails to do such a job, because there could be a vertex $v \in V(S^+) \setminus V(S)$
such that the other vertices of $S^+$ hide $v$ from the boundary (cf. Figure~\ref{F:hiddenvertex}). 
Each face of $S^+$, however, is incident to at least one vertex of $S$, hence faces of $S^+$
are not included in $D(T)$ and consequently no vertex in $V(S^+)$ lies in the interior of $D(T)$.
It is also not difficult to see that every vertex in $V(b_1 S^+)=V(bT)$ belongs to $V(S^+)$ since they are incident to faces of $S^+$.
Since $V(bT) \cap V(S) = \emptyset$ by the definition of $T$, we conclude that $V(b_1 S^+) = V(S^+) \setminus V(S)$. Also it is clear that 
the set of vertices enclosed by the walk $b_1 S^+$  is exactly $V(S)$; i.e., $b_1 S^+$ is the boundary walk we want. Now we are ready to provide our last
version of the combinatorial Gauss-Bonnet formula.

\begin{figure}[t]
\begin{center}
 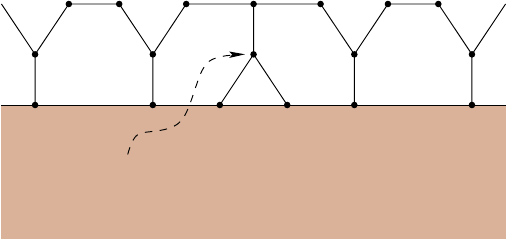
 \caption{The case that a vertex in $V(S^+) \setminus V(S)$ does not lie on $b_i S^+$.}\label{F:hiddenvertex}
\end{center}
\end{figure}

\begin{theorem}[Combinatorial Gauss-Bonnet Theorem-Type III]\label{CGBT3}
Suppose $G$ is an infinite tessellation, and let $S$ be a face subgraph of $G$ such that $D(S)^\circ$ is connected.
Also suppose that the induced subgraph $T$ of $G$ with $V(T) = V \setminus V(S)$ has $m$ components. 
Then we have
\begin{equation}\label{GBF-3a}
 \kappa(S) + \tau_{i} (b_1 S^+) = 2-m.
\end{equation}
\end{theorem}
\begin{proof}
Sharpening the corners of the closed $\epsilon$-neighborhood of $D(T)$ and applying \eqref{GBT-S} 
to the complement of this sharpened region  (where the complement is taken in $\mathcal{S}_G$),
one can easily check that \eqref{GBF-3a} follows. we leave the details to readers.
\end{proof}

In general the natural number $m$ in Theorem~\ref{CGBT3} cannot be described by $\chi (S)$, $\chi(D(S^+)^\circ)$, etc. See Figure~\ref{F:diffEuler}.
Also remark that some cycles in $b_1 S^+$ (i.e., $-\gamma_j$ for some $j =1,2, \ldots, m$, where $\gamma_j$ is a cycle in \eqref{E:gamma_j}) 
might be constant, say $[v]$ for some $v \in V$, for which the inner left turn was defined by $\tau_i ([v]) = \kappa(v) -1$. 

\begin{figure}[t]
\begin{center}
 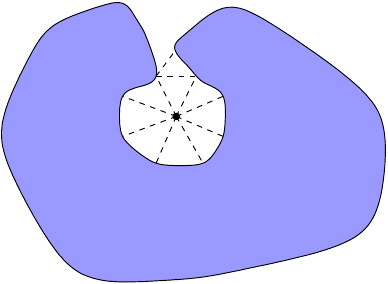
 \caption{The case where $m$, the number of components of $D(T)$, cannot be described in terms of the Euler characteristics of $S$, $D(S^+)^\circ$, etc. 
 In the above figure we assume that $\chi(S) = \chi (D(S^+)^\circ)=1$, but we have $m=2$.}\label{F:diffEuler}
\end{center}
\end{figure}

When the subgraph $S$  is either connected or a polygon,  statements similar to Theorem~\ref{CGBT2} (and perhaps  Theorem~\ref{CGBT1} as well)  can be found
in \cite{BP01, BP06, DeMo07, KP11}. This means that what is new in this section lies in that Theorems~\ref{CGBT1},  \ref{CGBT2}, and \ref{CGBT3} 
have been stated to cover various situations and the formulae have been described not by inequalities but by equalities. 
We also believe that, even though there are no serious proofs in this section, these theorems  are the most important parts throughout the paper.
Some of these theorems were also extensively used  to determine circle packing types of disk triangulation graphs \cite{Oh22}.

\begin{remark*}
One can seek combinatorial proofs for Theorems~\ref{CGBT1}, \ref{CGBT2},  and \ref{CGBT3}, for which we have used a geometric formula \eqref{GBT-S}. 
While combinatorial proofs may be possible, we believe that such proofs would involve many complicated cases. 
For instance, the theorems are described not only in terms of the graph's combinatorial properties but also its topological properties. 
Therefore, proving them solely through combinatorial methods is likely to be a challenging task.
\end{remark*}

\section{A lemma}\label{S:Lemma}
In this section we prove a lemma which will play an important role in the proof of Theorem~\ref{T1}. 
Let us first introduce some more concepts.

Suppose a tessellation of the plane $G=(V, E, F)$ is given, and let $S$ be a finite subgraph of $G$. Assume that 
$b_i S = \gamma_1 \cup \gamma_2 \cup \cdots \cup \gamma_m$ is the inner
boundary walk of $S$, where $\gamma_j$'s  are as explained in the previous section. For $j \in \{1,2, \ldots, m\}$, 
let  $\gamma_j = [v_0, v_1, \ldots, v_n=v_0]$ for some vertices $v_0, v_1, \ldots, v_{n-1}$.
Now if there exists an edge $e$ that is incident to $v_k$ and lies on the left of $[v_{k-1}, v_k, v_{k+1}]$, where we have the convention $v_{n+1} = v_1$ if $k =n$ as before, 
we will call such $e$ an \emph{inward} edge and denote by $\mathfrak{e}(v_k; v_{k-1}, v_{k+1})$ the number of  inward edges incident to $v_k$. We also define
$\mathfrak{e}(\gamma_j) = \sum_{k=1}^n \mathfrak{e}(v_k; v_{k-1}, v_{k+1})$ and $\mathfrak{e} (S) = \sum_{j=1}^m \mathfrak{e}(\gamma_j)$.
Similarly let $b S = \Gamma_1 \cup \cdots \cup \Gamma_{m'}$ and  $\Gamma_j = [w_0, \ldots, w_{n'} = w_0]$ for $j \in \{1,2, \ldots, m'\}$.
If $n' \geq 2$ and there exists an edge $e$ that is incident to $w_k$ and lies on the right of $[w_{k-1}, w_k, w_{k+1}]$, we call such $e$ an \emph{outward} edge and
denote by $\mathcal{E}(w_k; w_{k-1}, w_{k+1})$ the number of outward edges incident to $w_k$. We also let 
$\mathcal{E} (\Gamma_j) = \sum_{k=1}^{n'} \mathcal{E}(w_k; w_{k-1}, w_{k+1})$ in this case. If $n'=0$, that is, if $\Gamma_j  =[w_0]$ for some $w_0 \in V$,
then we define $\mathcal{E}(\Gamma_j)$ as the set of all edges incident to $w_0$; i.e., we define $\mathcal{E}(\Gamma_j) = |E( w_0)|= \deg w_0$. 
Finally let $\mathcal{E}(S) = \sum_{j=1}^{m'} \mathcal{E}(\Gamma_j)$. Note that $\mathcal{E}(S) = | \partial S|$ if $S$ is induced, but in general only
the inequality $\mathcal{E}(S) \geq | \partial S|$ holds.

\begin{lemma}\label{ML1}
Suppose $G=(V,E,F)$ is a plane graph such that $p \leq \deg v < \infty$ and $q \leq \deg f < \infty$ for all $v \in V$ and $f \in F$, where $p$ and $q$
are natural numbers satisfying $1/p + 1/q \leq 1/2$. Let $S$ be a finite nonempty subgraph of $G$ such that $\mathbb{R}^2 \setminus D(S)$ has only one component
and $S^{-}$ contains at least two vertices. Then we have
\begin{equation}\label{E:ML1}
 |V(b_i S)| \geq (pq -2p -2q)|V(S^-)|  + |V(b S^{-})| + 2q.
\end{equation}
\end{lemma}
\begin{proof}
To show \eqref{E:ML1}, we may assume that $S$ is connected by considering each components of $S$ separately. For example if $S$ has two connected components,
then we will have $4q$ instead of $2q$ in the last term of \eqref{E:ML1}, but anyway \eqref{E:ML1} will be true as long as we prove it for connected subgraphs. 
Thus $S$ will be regarded as a simply connected subgraph, because  $\mathbb{R}^2 \setminus D(S)$ was assumed to have only one component. 

For $v \in V(S^-)$ we have
\begin{align*}
\kappa(v) & = 1 - \frac{\deg v}{2} + \sum_{f \in F(v)} \frac{1}{\deg f} \leq 1 - \frac{\deg v}{2} +  \frac{\deg v}{q} \\
                 & = 1 - \frac{q-2}{2 q} \cdot \deg v \leq - \frac{pq - 2p -2q}{2q},
 \end{align*}
because $\deg v \geq p$ and $\deg f \geq q \geq 3$ for every $f \in F$. Thus we have
\begin{equation}\label{E1-1}
\kappa (S^{-}) \leq - \frac{pq - 2p -2q}{2q} \cdot |V(S^{-})|.
\end{equation}
We next let $b_i S = \gamma_1 \cup \cdots \cup \gamma_m$, where $\gamma_j$'s are cycles as defined before. Note that, since $S$ is simply connected, 
$m = \chi (D(S)^\circ) \geq 1$ and each $\gamma_j$ must be a nonconstant simple cycle enclosing a component of $D(S)^\circ$. 
Let $\gamma_j = [v_0, v_1, \ldots, v_n]$  for some $j \in \{1,2, \ldots, m\}$. 
Then for each $k \in \{1,2, \ldots, n \}$, using the notation $\mathfrak{e}_k = \mathfrak{e}(v_k; v_{k-1}, v_{k+1})$, we have from \eqref{innerLT}  that 
\[
\tau_i (v_k; v_{k-1}, v_{k+1}) \leq \frac{1}{2} - \frac{\mathfrak{e}_k +1}{2} + \frac{\mathfrak{e}_k +1}{q} = \frac{2-q}{2q} \cdot \mathfrak{e}_k+ \frac{1}{q},
\]
because $\deg f \geq q$ for every $f \in F$ and $(\mathfrak{e}_k+1)$ is the number of faces lying on the left of $[v_{k-1}, v_k, v_{k+1}]$.
Thus we get 
\begin{equation}\label{Temp1}
 \tau_i (\gamma_j) \leq \frac{|\gamma_j|}{q}- \frac{q-2}{2q} \cdot \mathfrak{e}(\gamma_j) \quad \mbox{and} \quad  
 \tau_i (b_i S) \leq \frac{|b_i S|}{q} -  \frac{q-2}{2q}  \cdot \mathfrak{e}(S).
\end{equation}
 
Here a naive approach is to set $|b_i S| = |V(b_i S)|$, but this is in general not true because a vertex may appear in two or more cycles in $\gamma_1, \ldots, \gamma_m$. 
Note that $| b_i S|$ denotes the length of the inner boundary walk, while $|V(b_i S)|$ denotes the number of vertices in $b_i S$. However, if a vertex $v$ appears in more than
one cycle, simply connectedness of $S$ implies that such $v$ must be a \emph{cut vertex} of $S$; i.e., $v$ is a vertex such that $S \setminus \{ v \}$ becomes disconnected. 
Now if  $v_1, v_2, \ldots, v_n$ are the vertices in $V(b_i S)$ that appear in $l_1, l_2, \ldots, l_n$ cycles, where each $l_j$ is at least $2$, then we must have
$(l_1 -1) + (l_2 -1) + \cdots + (l_n -1) +1 \leq m$ because removing $v_j$ from $S$ makes the number of components increase by $l_j -1$ and  the number of components 
of $S \setminus \{ v_1, v_2, \ldots, v_n \}$ cannot exceed  that of $D(S)^\circ$.  Therefore we have
\begin{equation}\label{E:compnumber}
\begin{aligned}
|b_i S| - |V(b_i S)| &  = (l_1 -1) + (l_2 -1) + \cdots + (l_n -1) \\
 &\leq m-1  = \chi (D(S)^\circ) -1,
\end{aligned}
\end{equation}
 hence from \eqref{Temp1} we obtain
\begin{equation}\label{Temp2}
 \tau_i (b_i S) \leq \frac{|b_i S|}{q} -  \frac{q-2}{2q} \mathfrak{e}(S) \leq \frac{|V(b_i S)|+\chi (D(S)^\circ) -1}{q} -  \frac{q-2}{2q} \cdot  \mathfrak{e}(S).
\end{equation}
Now \eqref{E1-1}, \eqref{Temp2}, and the Gauss-Bonnet formula of the second type \eqref{GBF-2} imply that
\begin{equation}\label{E1}
\begin{aligned}
 |V(b_i S)| & \geq \frac{pq - 2p -2q}{2}\cdot |V(S^{-})| + \frac{q-2}{2} \cdot  \mathfrak{e}(S) + (q-1) \chi (D(S)^\circ) +1 \\
                   & \geq \frac{pq - 2p -2q}{2}\cdot |V(S^{-})| + \frac{q-2}{2} \cdot  \mathfrak{e}(S) + q,
\end{aligned}
\end{equation}
because $\chi (D(S)^\circ) =m \geq 1$.
 
In the next step we apply the Gauss-Bonnet formula of the first type \eqref{GBF-1} to the subgraph $S^-$, the induced subgraph whose vertices are in $D(S)^\circ$, and obtain
\begin{equation}\label{GBF-*}
\kappa(S^-) + \tau_{o} (b S^-) = \chi (S^-).
\end{equation}
Let $bS^- = \Gamma_1 \cup \cdots \cup \Gamma_{m'}$ as before. Then because $S$ is simply connected, $m'$ should be the number of components of $S^-$
and we have $m' = \chi (S^{-})$. That is, simply connectedness of $S$ implies that $D(S^{-})$ does not have \emph{holes}, and 
each $\Gamma_j$ must be the cycle enclosing a component of $D(S^-)$. We also remark that some
of $\Gamma_j$'s could be constant cycles. 

Now we will \emph{pretend} that every face outside $S^-$ is of degree $q$. In other words, in \eqref{GBF-*} we replace $\deg f$ by $q$ for all $f \notin F(S^-)$ whenever they appear,
and we see that the equality \eqref{GBF-*} still holds as explained in Remark~\ref{R-face}. 
Then the inequality \eqref{E1-1} definitely holds even in our pretended computation. Next for some $j \in \{1,2, \ldots, m' \}$, 
we assume $\Gamma_j = [w_0, w_1, \ldots, w_{n'}]$ with $n' \geq 2$. By \eqref{outerLT} our pretended computation gives 
\[
  \tau_o (w_k; w_{k-1}, w_{k+1}) \approx   \frac{\mathcal{E}_k +1}{2} - \frac{\mathcal{E}_k +1}{q}  - \frac{1}{2} = \frac{q-2}{2q} \cdot \mathcal{E}_k - \frac{1}{q}
\]
with the notation $\mathcal{E}_k =  \mathcal{E}(w_k; w_{k-1}, w_{k+1})$. Here the symbol `$\approx$' means that we are doing pretended computation; i.e., 
the computation is performed after we have replaced $\deg f$ by $q$ for $f \notin F(S^-)$. Therefore we have
\begin{equation}\label{Temp3}
\tau_o (\Gamma_j) = \sum_{k=1}^{n'} \tau_o (w_k; w_{k-1}, w_{k+1}) \approx \frac{q-2}{2q} \cdot \mathcal{E} (\Gamma_j) - \frac{|\Gamma_j|}{q}.
\end{equation}
If $\Gamma_j = [w_0]$ for some $j$, then we have $|\Gamma_j| =0$ and get
\begin{equation}\label{Temp4}
\tau_o (\Gamma_j) = 1- \kappa(w_0) \approx \frac{\deg w_0}{2} - \frac{\deg w_0}{q} =  \frac{q-2}{2q} \cdot \mathcal{E} (\Gamma_j) - \frac{|\Gamma_j|}{q}.
\end{equation}
Now \eqref{E1-1}, \eqref{Temp3}, \eqref{Temp4}, and \eqref{GBF-*} yield 
\[
|b S^-| + q \cdot\chi (S^-) + \frac{pq - 2p -2q}{2}\cdot |V(S^{-})|- \frac{q-2}{2} \cdot \mathcal{E}(S^-) \leq 0.
\]

Let $m_0$ be the number of components of $S^{-}$ consisting of single vertices. Then definitely we have $m_0 \leq m' = \chi(S^{-})$. 
But because $|V(b S^{-})| \leq |b S^-|+ m_0$,  we get 
\[
|V(b S^{-})|+ q \leq  |b S^-|+ m_0 + q \leq |b S^-| + m' q   =|b S^-| + q \cdot\chi (S^-) 
\]
unless $S^{-}$ itself is a single vertex, which is the case we have excluded in the assumption. Therefore we obtain
\begin{equation}\label{E2}
|V(b S^{-})|+ q  + \frac{pq - 2p -2q}{2}\cdot |V(S^{-})|- \frac{q-2}{2} \cdot \mathcal{E}(S^-) \leq 0.
\end{equation}

To finish the proof of Lemma~\ref{ML1}, suppose $e$ is an outward edge from a vertex on $b S^{-}$. 
Then a priori $e$ has an end on $S^{-}$. If the other end of $e$ belongs to $S^{-}$, then $e$ must be an edge in $E(S^{-})$ because $S^{-}$ is an induced subgraph,
hence the edge $e$ cannot be an outward edge. This means that the other end of $e$ lies on $b_i S$, hence we must have  $\mathfrak{e}(S) \geq \mathcal{E}(S^-)$. 
See Figure~\ref{F:inoutedge} for the case $\mathfrak{e}(S) > \mathcal{E}(S^-)$. Now one can easily obtain \eqref{E:ML1}  
by adding \eqref{E1} to \eqref{E2}, and this completes the proof of Lemma~\ref{ML1}.
\end{proof}
\begin{figure}[t]
\begin{center}
 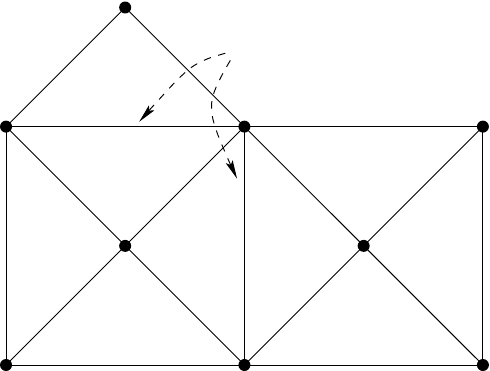
 \caption{A subgraph $S$ with $\chi(S^-) =2$, $\chi (D(S)^\circ) =1$, $\mathfrak{e}(S) =12$, and $\mathcal{E}(S^-)=8$.}\label{F:inoutedge}
\end{center}
\end{figure}

Note that the previous proof also shows that 
\begin{equation}\label{E:ML1-1}
 |V(b_i S)| \geq (pq -2p -2q)|V(S^-)|  + |V(b S^{-})| + 2q-1
\end{equation}
if $S^{-}$ is a single vertex, but it can also be verified easily using direct computation.

\section{Proof of Theorem~\ref{T1}}\label{ProofT1}
In this section we will prove Theorem~\ref{T1}, so throughout the section it is assumed that $G=(V,E,F)$ is a plane graph such that 
$p \leq \deg v < \infty$ and $q \leq \deg f < \infty$ for every $v \in V$ and $f \in F$, 
where $p$ and $q$ are natural numbers satisfying $1/p + 1/q \leq 1/2$.  But because there is nothing to prove when $1/p + 1/q = 1/2$, in which case we have $\Phi (p, q)= \Phi (q, p)  =0$,
we will further assume that $1/p + 1/q < 1/2$. Note that this condition is equivalent to
\begin{equation}\label{E:pq-cond}
pq -2p - 2q = (p-2) (q-2) - 4 >0.
\end{equation}

Let $S$ be a finite nonempty subgraph of $G$. Because we want to prove the inequality 
\begin{equation}\label{E:isof}
\frac{|b S|}{\sum_{f \in F(S)} \deg f} \geq  \frac{\Phi(q, p)}{q} =  \frac{q-2}{q} \sqrt{1 - \frac{4}{(p-2)(q-2)}},
\end{equation}
we may assume without loss of generality that $S$ is a simply connected face graph, since otherwise we can remove all the edges and vertices of $S$ which
are not incident to faces in $F(S)$, add to $S$ all the vertices, edges, and faces included in the closures of the bounded components of $\mathbb{R}^2 \setminus D(S)$,
and consider each components of $S$ separately.
That is, we may assume that $S$ is a polygon, and in this case the quantity $|b S|$ will be the same as the number of edges in $E(S)$ that are incident to faces in $F \setminus F(S)$.

First it is not difficult to see that
\[
\sum_{f \in F(S)} \deg f \leq 2|E(S)| - |bS|.
\]
Then because $S$ is simply connected and  $\deg f \geq q$ for all $f \in F$,  Euler's formula implies
\begin{align*}
1 & = |V(S)| - |E(S)| + |F(S)| \\ 
   & \leq |V(S)| - \frac{1}{2} \cdot \sum_{f \in F(S)}  \deg f - \frac{1}{2} \cdot|bS|+\frac{1}{q} \cdot  \sum_{f \in F(S)} \deg f,
\end{align*}
or we have
\begin{equation}\label{E:inequal-1}
\frac{q-2}{2q} \cdot \sum_{f \in F(S)} \deg f \leq |V(S)| - \frac{1}{2} \cdot |b S| -1.
\end{equation}

Let $N \geq 0$ be the depth of $S$.  (See the last paragraph in Section~\ref{prelim} for the definition of depth.) 
If $N=0$ then it is true that $|V(S)| \leq |bS|$, so \eqref{E:inequal-1} implies $(q-2)/q \leq |bS|/(\sum_{f \in F(S)} \deg f)$. Thus \eqref{E:isof} definitely holds in this case. Now let us assume that $N \geq 1$.

Define $S_ N = S$, and inductively we define $S_{N-k} = (S_{N-k+1})^-$ for $k =1,2, \ldots, N$. We also let $s_n = |V(b S_n)|$ for $n =0,1,2, \ldots, N$.
Then it is not difficult to see that $s_0 + s_1 + \cdots + s_n = |V(S_n)|$ for every $n \in \{0,1, \ldots, N \}$, because $(S_0)^- = \emptyset$ by the definition of depth.
Moreover, for $n \geq 2$ the graph $S_n$ must satisfy the assumptions of Lemma~\ref{ML1}, because we have assumed that $S$ is a simply connected face graph.
Now since $V(b_i T) \subset V(b T)$ for every subgraph $T \subset G$, 
\eqref{E:ML1} and \eqref{E:ML1-1} imply that
\begin{equation}\label{s_n}
\begin{cases}
s_n \geq s_{n-1} +  (pq -2p -2q) ( s_0  + \cdots + s_{n-1}) + 2q \quad \mbox{ for } n = 2, \ldots, N,\\
s_1 \geq s_0 + (pq -2p -2q) s_0 + 2q -1, \\
s_0 \geq 1.
\end{cases}
\end{equation}

We will compare $\{ s_n \}_{n=0}^N$ with the sequence $\{ a_n \}_{n=0}^N$ defined by 
\begin{equation}\label{a_n}
\begin{cases}
a_n = a_{n-1} +  (pq -2p -2q) ( a_0  + \cdots + a_{n-1}) + 2q \quad \mbox{ for } n = 2, \ldots, N,\\
a_1 = a_0 + (pq -2p -2q) a_0 + 2q -1, \\
a_0 =  t_0,
\end{cases}
\end{equation}
where $t_0$ is  determined as follows. 
Set $a_0 = t \in \mathbb{R}$, and observe that $a_1$ is a linear function in $t$ with positive coefficients.
Therefore $a_2 = a_1 + (pq -2 p -2q)(a_0 + a_1 ) +2q$ is also a linear function in $t$ with positive coefficients, and inductively we see that 
$a_N$ is a linear function in $t$ with positive coefficients as well. Thus 
there exists $t= t_0 \in \mathbb{R}$ that makes $a_N = s_N$, and we will use such $t_0$  in \eqref{a_n}.

If $t_0 = a_0 < s_0$, then definitely $a_n < s_n$ for every $n=1,2, \ldots, N$, contradicting our assumption $a_N = s_N$. Therefore $1 \leq s_0 \leq a_0$.
Suppose $$s_0 + s_1 + \cdots + s_{k-1} \leq a_0 + a_1 + \cdots + a_{k-1}$$ for some $k \in \{1, 2\ldots, N \}$. If
$$s_0 + s_1 + \cdots + s_{k} > a_0 + a_1 + \cdots + a_{k},$$ then definitely we have $s_k > a_k$. Therefore \eqref{s_n} and \eqref{a_n} imply that
$s_{k+1} > a_{k+1}$ and $s_0 + s_1 + \cdots + s_{k+1} > a_0 + a_1 + \cdots + a_{k+1}$. By repeating this argument we obtain the inequality $s_N > a_N$,
which contradicts our choice of $t_0$. Therefore we have 
\begin{equation}\label{induction}
s_0 + s_1 + \cdots + s_{k} \leq a_0 + a_1 + \cdots + a_{k},
\end{equation}
and by induction we see that \eqref{induction} is true for all $k =0,1,2, \ldots, N-1$, especially for the case $k = N-1$.
 
Let $P = p-2$ and $Q= q-2$, and note that $PQ -4 > 0$ by \eqref{E:pq-cond}. Also we see from \eqref{a_n} that the sequence $\{ a_n \}$ satisfies the recurrence relation
\begin{equation}\label{recur}
a_n  - (PQ -2)a_{n-1}  + a_{n-2} =0
\end{equation}
for $n \geq 3$. Let $\alpha = \frac{1}{2} \left\{ PQ -2 + \sqrt{(PQ-2)^2 -4} \right\} > 1$, and we observe that $\alpha$ and $1/\alpha$ are the zeros of the
characteristic polynomial $x^2 - (PQ-2) x + 1$ of the recurrence relation \eqref{recur}.  Here our goal is to show the inequality
\begin{equation}\label{inequal-2}
a_0 + a_1 + \cdots + a_{N-1} \leq \frac{a_N}{\alpha-1}.
\end{equation}
Suppose \eqref{inequal-2} is true. Then from \eqref{E:inequal-1} and \eqref{induction} we get
\begin{align*}
\frac{q-2}{2q} & \cdot \sum_{f \in F(S)} \deg f  \leq |V(S)| - \frac{1}{2} \cdot |b S| -1 \leq s_0 + s_1 + \cdots + s_N - \frac{s_N}{2}  \\
    & =  (s_0 + s_1 + \cdots + s_{N-1}) + \frac{s_N}{2} \leq  (a_0 + a_1 + \cdots + a_{N-1} ) + \frac{a_N}{2}  \\
    & \leq \frac{a_N}{\alpha-1} + \frac{a_N}{2}  = \frac{1}{2} \cdot \frac{\alpha +1}{\alpha -1} \cdot s_N \leq \frac{1}{2} \cdot \frac{\alpha +1}{\alpha -1} \cdot |b S|.
\end{align*}
Note that $\alpha^2 + 1 = (PQ -2) \cdot \alpha$ and $\alpha - \alpha^{-1} = \sqrt{(PQ-2)^2 -4}$ because $\alpha$ and $1/\alpha$ are the zeros of the polynomial
$x^2 - (PQ-2) x + 1$. Thus we have  
\begin{align*}
\frac{|b S|}{\sum_{f \in F(S)} \deg f} & \geq \frac{q-2}{q} \cdot \frac{\alpha -1}{\alpha+1} =  \frac{q-2}{q} \cdot \frac{\alpha^2 -1}{(\alpha+1)^2} 
     = \frac{q-2}{q} \cdot \frac{\alpha (\alpha - \alpha^{-1})}{\alpha^2 +1 + 2 \alpha} \\
    & = \frac{q-2}{q} \cdot \frac{\alpha \cdot \sqrt{(PQ-2)^2 -4}}{PQ \cdot \alpha} 
     = \frac{q-2}{q} \sqrt{\frac{(PQ)\cdot (PQ-4)}{(PQ)^2}}\\
    & =   \frac{q-2}{q} \sqrt{1 - \frac{4}{(p-2)(q-2)}},
\end{align*}
so \eqref{E:isof} follows.

Since $\alpha$ and $1/\alpha$ are the zeros of the polynomial $x^2 - (PQ-2) x + 1$, we have from \eqref{recur} that
\begin{equation}\label{recur2}
a_n - \alpha \cdot a_{n-1} = \frac{1}{\alpha} ( a_{n-1} - \alpha \cdot a_{n-2})
\end{equation}
for $n \geq 3$, and we obtain from \eqref{a_n}  that
\begin{equation}\label{recur3}
a_2 - \alpha \cdot a_{1} = \frac{1}{\alpha} ( a_{1} - \alpha \cdot a_{0})+1.
\end{equation}
We claim that 
\begin{equation}\label{E:claim}
a_n  \leq \alpha \cdot a_{n-1} + 2 \alpha q
\end{equation}
 for all $n  = 1,2, \ldots$. In fact, the second row of \eqref{a_n} can be read as
\begin{equation}\label{E:casen=1}
a_1 = a_0 + \left( \alpha + \frac{1}{\alpha} -2 \right) a_0 + 2q -1 = \left( \alpha + \frac{1}{\alpha} -1 \right) a_0 + 2q -1,
\end{equation}
hence we have
\[
a_1 - \alpha \cdot a_0 = \left( \frac{1}{\alpha} -1 \right) a_0 + 2q -1 \leq 0 \cdot a_0 + 2q -1 = 2q -1\leq 2 \alpha q
\]
because $\alpha \geq 1$ and $a_0 \geq 1>0$. Therefore \eqref{recur3} implies
\[
a_2 - \alpha \cdot a_1 = \frac{1}{\alpha} ( a_{1} - \alpha \cdot a_{0})+1 \leq \frac{2q-1}{\alpha} +1 \leq (2q -1) +1 = 2q \leq 2 \alpha q.
\]
For $n \geq 3$ we have from \eqref{recur2} that
\begin{align*}
a_n -  & \alpha \cdot a_{n-1}  = \frac{1}{\alpha} ( a_{n-1} - \alpha \cdot a_{n-2}) = \frac{1}{\alpha^2} ( a_{n-2} - \alpha \cdot a_{n-3}) \\
                                             & = \cdots = \frac{1}{\alpha^{n-2}} ( a_{2} - \alpha \cdot a_{1}) \leq   \frac{2 \alpha q}{\alpha^{n-2}} \leq 2 \alpha q,
\end{align*}                                             
so the claim has been proved.

If $N=1$, the inequality \eqref{inequal-2} easily comes from \eqref{E:casen=1}.
If $N \geq 2$, the first row of  \eqref{a_n} can be written as 
\[
\left( \alpha + \frac{1}{\alpha} -2 \right) (a_0 + a_1 + \cdots + a_{N-1}) = a_N - a_{N-1} -2q.
\]
But the claim \eqref{E:claim} implies that
\[      
 a_N - a_{N-1} -2q  \leq a_N - \frac{a_N}{\alpha}  + \frac{2 \alpha q}{\alpha}  -2q     = \frac{\alpha -1}{\alpha} \cdot a_N,
\]
so we have
\begin{align*}
a_0 + a_1 + & \cdots + a_{N-1} \leq      \left( \alpha + \frac{1}{\alpha} -2 \right) ^{-1} \cdot     \frac{\alpha -1}{\alpha} \cdot  a_N \\
& = \frac{\alpha}{(\alpha-1)^2} \cdot \frac{\alpha -1}{\alpha} \cdot  a_N = \frac{a_N}{\alpha -1},
\end{align*}               
which proves \eqref{inequal-2}. We conclude that \eqref{E:isof} holds for every $S \subset G$ such that $F(S) \neq \emptyset$.

\begin{proof}[Proof of Theorem~\ref{T1}]
Since the inequality \eqref{E:isof} holds for every nonempty finite subgraph $S$ with $F(S) \neq \emptyset$, we definitely  have  $\imath_\sigma ^*(G) \geq \Phi (q, p)/q$.
The inequality $\imath_\sigma (G)  = \imath_\sigma ^*(G^*) \geq \Phi (p,q)/p$ comes from the duality. Finally because $\sum_{v \in V(S)} \deg v \geq p \cdot |V(S)|$ and
 $\sum_{v \in F(S)} \deg f \geq q \cdot |F(S)|$, it is easy to see that $\imath (G) \geq \Phi (p, q)$ and $\imath^* (G) \geq \Phi(q, p)$. This completes the proof of Theorem~\ref{T1}.
\end{proof}

\section{Proof of Theorem~\ref{T2}}\label{ProofT2}
In this section we will prove Theorem~\ref{T2}, and our strategy is to follow each steps in Sections~\ref{S:CGB}--\ref{ProofT1} with inequalities reversed. 
One can check that our proof of Theorem~\ref{T2} is almost the same as that of Theorem~\ref{T1}, but there are some details we need to treat carefully.
 
Suppose $G=(V,E,F)$ is an infinite tessellation of the plane such that $\deg v \leq p$ and $\deg f \leq q$ for all $v \in V$ and $f \in F$, where $p$ and $q$
are natural numbers satisfying $1/p + 1/q \leq 1/2$. Choose $f_0 \in F$, and for $n \in \mathbb{N}$ let $\mathcal{B}_n := \mathcal{B} (f_0)$ be the quasi-ball with the core $f_0$ and height $n$. 
(See Section~\ref{prelim} for the definition of quasi-balls.) We also set  $\mathcal{B}_0 = f_0$, the face graph with $F(\mathcal{B}_0) = \{ f_0 \}$. 
Now we fix $n \in \{0 \} \cup \mathbb{N}$, and note that $\chi (\mathcal{B}_n) \leq 1$ because $\mathcal{B}_n$ is connected. Moreover we can write 
$b \mathcal{B}_n = \Gamma_1 \cup \Gamma_2 \cup \cdots \cup \Gamma_m$  with $m = 2- \chi (\mathcal{B}_n)$, where each $\Gamma_j$ 
corresponds to the boundary of a component of $\mathbb{R}^2 \setminus D(\mathcal{B}_n)$ as before. Definitely each $\Gamma_j$ is not constant, and it must be  simple  
because $\mathcal{B}_n$ is a face graph and $D(\mathcal{B}_n)$ has a connected interior; i.e., $\mathcal{B}_n$ has no cut vertex. 
See (a) and (b) of Figure~\ref{F:cannotheppen} for the cases which cannot happen for $\mathcal{B}_n$.

Suppose $\Gamma_j = [v_0, v_1, \ldots, v_l= v_0]$, and we denote by $\mathcal{E}_k = \mathcal{E}(v_k; v_{k-1}, v_{k+1})$ the number of edges incident to
$v_k$ and lying on the right of $[v_{k-1}, v_k, v_{k+1}]$.  Next we will do the \emph{pretended computation} as in the proof of Theorem~\ref{T1}. 
That is, we apply \eqref{GBF-1} to $\mathcal{B}_n$ and obtain
\begin{equation}\label{GBF-**}
\kappa (\mathcal{B}_n ) +  \tau_o (b \mathcal{B}_n) = \chi (\mathcal{B}_n).
\end{equation}
Then we replace $\deg f$ by $q$ for all $f \notin F(\mathcal{B}_n )$  and observe that the equality \eqref{GBF-**} remains true  as explained in Remark~\ref{R-face}.
That is, we can pretend that every face outside $\mathcal{B}_n$ is of degree $q$. Then our pretended computation gives
\[
\tau_o ( v_k; v_{k-1}, v_{k+1}) \approx \frac{\mathcal{E}_k +1}{2} - \frac{\mathcal{E}_k +1}{q} - \frac{1}{2} = \frac{q-2}{2q} \cdot \mathcal{E}_k - \frac{1}{q},
\]
hence we obtain
\[
 \tau_o (b \mathcal{B}_n) \approx \frac{q-2}{2q} \cdot \mathcal{E}(\mathcal{B}_n) - \frac{|b \mathcal{B}_n|}{q}
\]
with the notation $\mathcal{E}(\Gamma_j) = \sum_{k=1}^l \mathcal{E}(v_k; v_{k-1}, v_{k+1})$ and 
$\mathcal{E} (\mathcal{B}_n) = \sum_{j=1}^m \mathcal{E}(\Gamma_j)$.
On the other hand, for $v \in V(\mathcal{B}_n)$ we have 
\[
\kappa (v) = 1 - \frac{\deg v}{2} + \sum_{f \in F(v)} \frac{1}{\deg f} \gtrsim 1 - \frac{q-2}{2q} \cdot \deg v \geq - \frac{pq -2 p -2q}{2q},
\]
 where the notation `$\gtrsim$' means that we have  inequality `$\geq$' in our pretended computation. Therefore
\[
\kappa (\mathcal{B}_n ) \gtrsim - \frac{pq -2 p -2q}{2q} \cdot |V(\mathcal{B}_n)|,
\]
and \eqref{GBF-**} implies
\begin{equation}\label{EE1}
q \cdot \chi (\mathcal{B}_n) \geq - \frac{pq -2p -2q}{2} \cdot |V(\mathcal{B}_n)| + \frac{q-2}{2} \cdot \mathcal{E}(\mathcal{B}_n) - |b \mathcal{B}_n|.
\end{equation}
\begin{figure}[t]
\begin{center}
\subfigure[]{\epsfig{figure=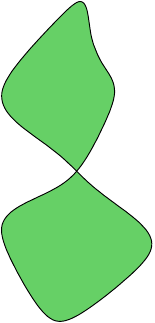, width=1. in}} \hspace{1 cm}
\subfigure[]{\epsfig{figure=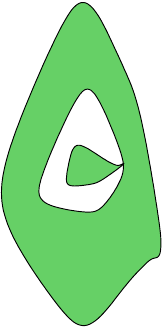, width=1. in}} \hspace{1 cm}
\subfigure[]{\epsfig{figure=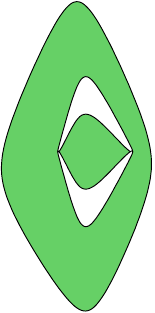, width=1. in}} \\
\subfigure[]{\epsfig{figure=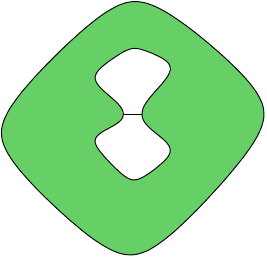, width=1.5 in}} \hspace{1.2 cm}
\subfigure[]{\epsfig{figure=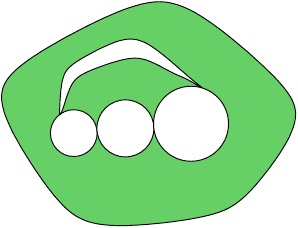, width=1.5 in}} 
\caption{The cases that cannot happen for $\mathcal{B}_n$.}\label{F:cannotheppen}
\end{center}
\end{figure}

Recall that each $\Gamma_j$ is simple and corresponds to the boundary of a component of $\mathbb{R}^2 \setminus D(\mathcal{B}_n)$. 
Then since  the interior of $D(\mathcal{B}_n)$ is connected,
we see that $\Gamma_j \cap \Gamma_{j'}$ has at most a component for $j \ne j'$. See Figure~\ref{F:cannotheppen}(c) which cannot happen for $\mathcal{B}_n$.
But if $\Gamma_j \cap \Gamma_{j'} \ne \emptyset$, then the intersection must be a vertex since $\mathcal{B}_n$ is a face graph and every edge in $E(\mathcal{B}_n)$ is 
incident to  at least one face in $F(\mathcal{B}_n)$ (Figure~\ref{F:cannotheppen}(d)). Furthermore, one can check that the above argument can be extended as follows:  
if $\Gamma_{j_1} \cup \Gamma_{j_2} \cup \cdots \cup \Gamma_{j_t}$ is connected and
$(\Gamma_{j_1} \cup \Gamma_{j_2} \cdots \cup \Gamma_{j_t}) \cap \Gamma_{j'} \ne \emptyset$ for some $j' \notin \{ j_1, j_2, \ldots, j_t \}$,
then the intersection has only one component and it must be a vertex (Figure~\ref{F:cannotheppen}(e)).
We conclude that if a vertex $v$ appears in $l$ distinct cycles in $\Gamma_1, \Gamma_2, \ldots, \Gamma_m$, then 
subtracting $v$ from $D(\mathcal{B}_n)$ makes the number of components of $\mathbb{R}^2 \setminus D(\mathcal{B}_n)$ decrease by $l-1$, hence we have 
\begin{equation}\label{bnchi}
|b \mathcal{B}_n | - |V (b \mathcal{B}_n) | \leq m -1 = 2 - \chi(\mathcal{B}_n) -1 = 1 - \chi(\mathcal{B}_n)
\end{equation}
as we did in \eqref{E:compnumber}. Consequently we deduce from \eqref{EE1} the inequality
\begin{equation}\label{E5}
0  \leq  |V(b \mathcal{B}_n)| +  q + \frac{pq - 2p -2q}{2}\cdot |V(\mathcal{B}_n)|- \frac{q-2}{2} \cdot \mathcal{E}(\mathcal{B}_n)
\end{equation}
because $\chi (\mathcal{B}_n) \leq 1$.

In the next step we need an inequality similar to \eqref{E1}. Note that $\mathcal{B}_n$ is a face subgraph such that $D(\mathcal{B}_n)^\circ$ is connected. Since $\mathcal{B}_{n+1} = \mathcal{B}_n^+$, the Gauss-Bonnet formula of the third type \eqref{GBF-3a} can be read as 
\begin{equation}\label{GBF-3}
 \kappa(\mathcal{B}_n) + \tau_{i} (b_1 \mathcal{B}_{n+1}) = 2-m',
\end{equation}
where $m'$ is the number of components of the induced subgraph $\mathcal{T}_n$ satisfying $V(\mathcal{T}_n) = V \setminus V(\mathcal{B}_n)$.
Also we have $V(b_1 \mathcal{B}_{n+1})=V(\mathcal{B}_{n+1}) \setminus V(\mathcal{B}_{n}) $, for which we have introduced the boundary walk $b_1 \mathcal{B}_{n+1}$.

Now we perform  computation  similar to that in Section~\ref{S:Lemma}, and we have
\[
\kappa (\mathcal{B}_n ) \geq - \frac{pq -2 p -2q}{2q} \cdot |V(\mathcal{B}_n)|
\]
and 
\[
\tau_i (b_1 \mathcal{B}_{n+1}) \geq - \frac{q-2}{2q} \cdot \mathfrak{e}(\mathcal{B}_{n+1}) + \frac{|b_1 \mathcal{B}_{n+1}|}{q},
\]
where $\mathfrak{e}(\mathcal{B}_{n+1}) $ denotes the number of inward edges from $b_1 \mathcal{B}_{n+1}$. 
Also note that at least one component of $\mathcal{T}_n$ is different from a vertex, hence at least one cycle consisting of $b_1 \mathcal{B}_{n+1}$  is not constant. Therefore we have
\[
|b_1 \mathcal{B}_{n+1}| + (m' -1) \geq |V(b_1 \mathcal{B}_{n+1})|.
\]
But since
\[
q(2-m') + m' -1 = q + (m'-1)(1-q) \leq q,
\]
we obtain from \eqref{GBF-3} that 
\begin{equation}\label{E6}
|V(b_1 \mathcal{B}_{n+1})| \leq \frac{pq-2p -2q}{2} \cdot |V(\mathcal{B}_n)| + \frac{q-2}{2} \cdot \mathfrak{e}(\mathcal{B}_{n+1})  + q.
\end{equation}
Finally note that $\mathfrak{e}(\mathcal{B}_{n+1}) \leq  \mathcal{E}(\mathcal{B}_n)$ because $\mathcal{T}_n$ is induced; i.e., each inward edge from $b_1 \mathcal{B}_{n+1}$ must be an outward edge
from $b \mathcal{B}_n$. Thus by adding \eqref{E5} to \eqref{E6} we get
\begin{equation}\label{E7}
|V(b_1 \mathcal{B}_{n+1})| \leq (pq-2p -2q) |V(\mathcal{B}_n)| + |V(b \mathcal{B}_n)| + 2q,
\end{equation}
which is what we have wanted to derive.

Suppose  $1/p + 1/q = 1/2$. Then \eqref{E7} can be read as $$|V(b_1 \mathcal{B}_{n+1})| \leq |V(b_1 \mathcal{B}_n)| + 2q$$ because $pq-2p-2q=0$ and
$ V(b \mathcal{B}_n) \subset V(b_1 \mathcal{B}_n)$ for all $n =0,1,2, \ldots$. Then since the combinatorial balls $B_n := B_n (v_0)$ are included in 
$\mathcal{B}_n = \mathcal{B}_n (f_0)$ for some $v_0 \in V(f_0)$,  the number of vertices in  $B_n$ grows at most polynomially. 
But in $G$ the vertex degrees are uniformly bounded by $p$, so if the inequality $\imath(G) >0$ held then combinatorial balls should grow exponentially 
(cf.\ see \eqref{E:ball-sphere} in Section~\ref{S:triangulation}, and the explanation of this phenomenon there). Therefore we must have $\imath (G)=0$, and
it is also not difficult to show, using duality and simple computation, that  all the other isoperimetric constants  in \eqref{E:isoconst} are zero. 
Thus the statements in Theorem~\ref{T2} follow in this case. Now we will assume that $1/p + 1/q < 1/2$ for the rest of this section.

Set $b_0 = |V(\mathcal{B}_0)| = \deg f_0$ and $b_n = |V(b_1 \mathcal{B}_{n})|$ for $n \in \mathbb{N}$. Then because $|V(\mathcal{B}_n) |= b_0 + b_1 + \cdots + b_n$ 
and $|V(b \mathcal{B}_n)| \leq |V(b_1 \mathcal{B}_{n})|= b_n$, we have from \eqref{E7} that
\[
\begin{cases}
b_n \leq b_{n-1} +  (pq -2p -2q) ( b_0  + \cdots + b_{n-1}) + 2q \quad \mbox{ for } n \in \mathbb{N},\\
b_0 \leq q.
\end{cases}
\]
Let $\{ a_n \}$ be the sequence satisfying
\begin{equation}\label{a_n2}
\begin{cases}
a_n = a_{n-1} +  (pq -2p -2q) ( a_0  + \cdots + a_{n-1}) + 2q \quad \mbox{ for } n \in \mathbb{N},\\
a_0 = t_0,
\end{cases}
\end{equation}
where $t_0$ is a real number which makes $a_N = b_N$ for some sufficiently large $N$. Here $N$ is to be determined later, and beware that $t_0$ could be a negative number. 
We also set $P=p-2$, $Q=q-2$, and $\alpha = \frac{1}{2} \left\{ PQ -2 + \sqrt{(PQ-2)^2 -4} \right\} > 1$.

It is clear that $a_0 \leq b_0$, since otherwise we would have the contradiction $a_N > b_N$. Then one can prove, using mathematical induction as in the previous section, that
\begin{equation}\label{order}
a_0 + a_1 + \cdots + a_k \leq b_0 + b_1 + \cdots + b_k
\end{equation}
for all $k=0,1, \ldots, N-1$. We next claim that
\begin{equation}\label{E:claimx}
a_n \geq \alpha \cdot a_{n-1}
\end{equation}
for all $n$. Note that \eqref{a_n2} implies
\[
a_1 = a_0 + \left( \alpha + \frac{1}{\alpha} -2 \right) a_0 + 2q,
\]
or
\[
a_1 - \alpha \cdot a_0 = \left( \frac{1}{\alpha} -1 \right) a_0 + 2q \geq  \left( \frac{1}{\alpha} -1 \right) q + 2q \geq q >0
\]
because $a_0 \leq b_0 \leq q$ and $\alpha > 1$. Thus the claim \eqref{E:claimx} easily follows since $a_n$ satisfies the equation \eqref{recur2}; i.e., it follows because
\[
a_n - \alpha \cdot a_{n-1} = \frac{1}{\alpha} ( a_{n-1} - \alpha \cdot a_{n-2})
\]
for all $n \geq 2$. Therefore by \eqref{a_n2} and \eqref{E:claimx} we have
\[
 \left( \alpha + \frac{1}{\alpha} -2 \right)(a_0 + a_1 + \cdots + a_{n-1}) =  a_n -a_{n-1} - 2q \geq  \left( 1 - \frac{1}{\alpha} \right)  a_n -2q,
\]
 or
\begin{equation}\label{E9}
a_0 + a_1 + \cdots + a_{n-1} \geq \frac{a_n}{\alpha -1} - \frac{2q\alpha}{(\alpha-1)^2}.
\end{equation} 

Let $\epsilon>0$ be given, and we choose $N$ such that
\begin{equation}\label{N}
\sum_{f \in F(\mathcal{B}_N)} \deg f > \frac{1}{\epsilon} \left( \frac{4 q \alpha}{\alpha^2 -1} + \frac{2(\alpha-1)}{\alpha +1} \right).
\end{equation}
Because $\mathcal{B}_N$ is a face graph, we have
\begin{equation}\label{Eul}
\sum_{f \in F(\mathcal{B}_N)}\deg f = 2|E(\mathcal{B}_N)| - |b \mathcal{B}_N|.
\end{equation}
 Then since $q \cdot |F(\mathcal{B}_n)| \geq \sum_{f \in F(\mathcal{B}_n)}\deg f$, Euler's formula and \eqref{Eul} yield
\[
\begin{aligned}
\chi (\mathcal{B}_N)  & = |V(\mathcal{B}_N)| - |E(\mathcal{B}_N)| + |F(\mathcal{B}_N)| \\
                                      & \geq |V(\mathcal{B}_N)| - \frac{1}{2} \cdot \sum_{f \in F(\mathcal{B}_N)}\deg f - \frac{1}{2} \cdot |b \mathcal{B}_N| + \frac{1}{q} \cdot \sum_{f \in F(\mathcal{B}_N)}\deg f \\
\end{aligned}
\]
or
\begin{equation}\label{F1}
\frac{q-2}{2q} \cdot\sum_{f \in F(\mathcal{B}_N)}\deg f \geq  |V(\mathcal{B}_N )|- \frac{|b \mathcal{B}_N|}{2} -\chi(\mathcal{B}_N).
\end{equation}
Note that $|V(\mathcal{B}_N )| = b_0 + b_1 + \cdots + b_N$ and $a_N = b_N$. We also have $b_N \geq |b \mathcal{B}_N| -1 + \chi (\mathcal{B}_N)$ from \eqref{bnchi},
because $b_N = |V(b_1 \mathcal{B}_N)| \geq |V(b \mathcal{B}_N)|$. 
Thus  we have from \eqref{order}, \eqref{E9}, and \eqref{F1} that 
\begin{align*}
\frac{q-2}{2q} \cdot & \sum_{f \in F(\mathcal{B}_N)}\deg f \geq (b_0 + b_1 + \cdots + b_{N-1}) + b_N -  \frac{|b \mathcal{B}_N|}{2} -\chi(\mathcal{B}_N) \\
 &  \geq  (a_0 + a_1 + \cdots + a_{N-1}) +  \frac{|b \mathcal{B}_N|}{2} -1 \geq \frac{a_N}{\alpha-1} +  \frac{|b \mathcal{B}_N|}{2} -1-  \frac{2q\alpha}{(\alpha-1)^2}\\
                                     & \geq \frac{|b \mathcal{B}_N| -1 + \chi (\mathcal{B}_N)}{\alpha-1} +  \frac{|b \mathcal{B}_N|}{2} -1-  \frac{2q\alpha}{(\alpha-1)^2}.
\end{align*}

We need to get rid of $\chi (\mathcal{B}_N)$ from the above inequality. For this purpose, we consider  the subgraph $\mathcal{A}_N \subset G$ which is obtained  
by adding to $\mathcal{B}_N$ all the vertices, edges, and faces included in the closures of the bounded components of  $\mathbb{R}^2 \setminus D(\mathcal{B}_N)$. 
That is, we fill the holes of $\mathbb{R}^2 \setminus D(\mathcal{B}_N)$ to obtain $\mathcal{A}_N$.  Then because the number of bounded components of
$\mathbb{R}^2 \setminus D(\mathcal{B}_N)$ is $1- \chi(\mathcal{B}_N)$, a sloppy estimate yields 
$|b \mathcal{B}_N| \geq |b \mathcal{A}_N| +1 -  \chi(\mathcal{B}_N) \geq  |b \mathcal{A}_N|$ and
$\sum_{f \in F(\mathcal{B}_N)}\deg f  \leq \sum_{f \in F(\mathcal{A}_N)}\deg f$. 
Thus we have
\begin{align*}
\frac{q-2}{2q} \cdot & \sum_{f \in F(\mathcal{A}_N)}\deg f  \geq \frac{|b \mathcal{A}_N|}{\alpha-1} +  \frac{|b \mathcal{A}_N|}{2} -1-  \frac{2q\alpha}{(\alpha-1)^2}\\
& = \frac{\alpha+1}{2(\alpha-1)} \cdot |b \mathcal{A}_N|  -  \frac{2q\alpha}{(\alpha-1)^2} -1.
\end{align*}
Now \eqref{N} implies that 
\begin{equation}\label{squid}
\frac{|b \mathcal{A}_N |}{ \sum_{f \in F(\mathcal{A}_N)}\deg f } \leq \frac{\Phi(q, p)}{q} + \epsilon  =  \frac{q-2}{q} \sqrt{1 - \frac{4}{(p-2)(q-2)}} + \epsilon.
\end{equation}
We conclude that $\imath_\sigma (G) \leq \Phi(q,p)/q$ because $\epsilon >0$ is arbitrary.
Finally the other inequalities in Theorem~\ref{T2} can be easily deduced from \eqref{squid}, using duality and simple computation, 
and we left them to readers. This completes the proof of Theorem~\ref{T2}.

\section{Vertex isoperimetric constants}\label{S:triangulation}
Our bounds for isoperimetric constants in Theorems~\ref{T1} and \ref{T2} were obtained by considering quasi-balls $\mathcal{B}_n$ instead of combinatorial balls $B_n$. 
In the proof of Theorem~\ref{T2} we directly  used quasi-balls, and our proof of Theorem~\ref{T1} was performed with quasi-balls in mind throughout. 
For example, we have come up with the idea of using \emph{depth}  because we wanted to study quasi-balls.  Moreover,  it seems that combinatorial balls never give isoperimetric constants 
as explained in \cite[Remark~4.3]{HJL}.  But there are still lots of problems stated in terms of combinatorial balls, and our method using quasi-balls does not work 
for such problems in general.  For example, our method does not give any solution for finding  the \emph{exponential growth rate}  defined by
\[
\mu(G) := \limsup_{n \to \infty} \frac{1}{n} \ln |V(B_n)|.
\]

On the other hand, in literature there is another type of isoperimetric constants other than those in \eqref{E:isoconst}. 
Let $S$ be a finite subgraph of a tessellation $G=(V, E, F)$ of the plane,
and let $d_0 S$ be the set of vertices in $V(S)$ which have neighbors in $V \setminus V(S)$. Also define $d_1 S$ as the set of vertices in $V \setminus V(S)$ which
have neighbors in $V(S)$. Then the \emph{vertex isoperimetric constants} are defined by
\[
\jmath_0 (G) = \inf_S \frac{|d_0 S|}{|V(S)|} \quad \mbox{and} \quad \jmath_1 (G) = \inf_S \frac{|d_1 S|}{|V(S)|},
\]
where  infima are taken over all finite nonempty subgraphs $S$.
The constants $\jmath_0 (G)$ and $\jmath_1 (G)$  are studied in geometric group theory \cite{CDP, GhyHar, HP, Gro}, and they also appear in many other literature concerning graphs
\cite{Alon, BS, Doz, Gel87, Gel, HJL, HP20, LP16, Moh02, Oh14, Oh15-2, Oh22, OS16, Soa90, Soa}. Note that these constants are related to each other
by the equation $\jmath_0 (G) =\jmath_1(G)/(1+ \jmath_1 (G))$   (cf.\ \cite[Lemma~6.1]{Oh22}). Moreover, one can  see that if $\jmath_0 (G)>0$ or $\jmath_1(G)>0$,
then all of the other isoperimetric constants $\imath (G), \imath_\sigma (G), \imath^* (G)$, and $\imath^*_\sigma (G)$ are positive (cf.\ \cite[Theorem~1(c)]{Oh14}. 
Conversely, if $\imath (G)>0$  and  vertex degrees are uniformly bounded (or $\imath^*(G)>0$ and face degrees are uniformly bounded),  we have $\jmath_0 (G)>0$ and $\jmath_1(G)>0$.

Like the exponential growth rate, our method using quasi-balls does not work well for finding exact values of  vertex isoperimetric constants.
In fact, the methods in previous sections could yield some bounds for $\jmath_0 (G)$ and $\jmath_1 (G)$, but one can check that these bounds are not sharp.
(We will describe such bounds in Theorem~\ref{T:triangulation} below, but only for triangulations.)
This phenomenon is understandable, however, because vertex isoperimetric constants are closely related to the exponential growth rate. 
Fix $v_0 \in V$, and let $B_n$ and $S_n$ be the combinatorial balls and spheres, respectively, with radius $n$ and centered at $v_0$. If $\jmath_1 (G)>0$, then
\begin{equation}\label{E:ball-sphere}
|V(B_n)| = |V(S_n)| + |V(B_{n-1})| \geq \jmath_1 (G)\cdot |V(B_{n-1})|+ |V(B_{n-1})| 
\end{equation}
for all $n \in \mathbb{N}$, hence we have $|V(B_n)| \geq (1 + \jmath_1 (G))^n$ and 
\begin{equation}\label{E:growth-iso}
\mu(G) \geq \ln (1 + \jmath_1 (G)) = \ln \left(\frac{1}{1 - \jmath_0 (G)} \right).
\end{equation}

For triangulations, however, we have a different story because quasi-balls with a vertex core become  combinatorial balls. For example, out computation immediately 
yields  the following statement.

\begin{theorem}\label{T:triangulation}
Suppose $G = (V, E, F)$ is a triangulation of the plane and let $p \geq 6$. If $\deg v \geq p$ for all $v \in V$ then
\[
\jmath_1 (G) \geq \frac{(p-6)+\sqrt{(p-2)(p-6)}}{2},
\]
and if $\deg v \leq p$ for all $v \in V$ then
 \[
\jmath_1 (G) \leq \frac{(p-6)+\sqrt{(p-2)(p-6)}}{2}.
\]
\end{theorem}
\begin{proof}
Suppose $\deg v \geq p$ for all $v \in V$, and let $S \subset G$ be a finite \emph{induced} subgraph of $G$. 
We claim that every vertex in $b S$ belongs to $d_0 S$. In fact, if $v \in V(b S) \setminus d_0 S$, 
then there exists $u, w \in V(b S)$ such that a sequence of the form $[\ldots, u, v, w, \ldots]$ appears on $b S$ but there is no edge incident to $v$ and lying on the right of $[u, v, w]$. But
because $G$ is a triangulation, the face incident to $v$ and lying on the right of $[u, v, w]$ must be the triangle with vertices $u, v$, and $w$, hence the edge $[u, w]$ belongs to $E(S)$
and the sequence  $[u, v, w]$ cannot appear on $b S$. This contradiction proves the claim. 

Now we follow the proof of Theorem~\ref{T1}, and obtain from \eqref{inequal-2} that 
\[
|d_0 S| \geq \left( 1 - \frac{1}{\alpha} \right) \cdot |V(S)| \quad \mbox{and} \quad |d_1 S| \geq (\alpha -1) \cdot |V(S)|, 
\]
where $\alpha = \frac{1}{2} \left\{ p-4 + \sqrt{(p-4)^2 -4} \right\}$. This completes the proof of the first statement of Theorem~\ref{T:triangulation}. If $\deg v \leq p$ for all $v \in V$,
we follow Theorem~\ref{T2} and obtain from \eqref{E9} the second statement of Theorem~\ref{T:triangulation}. We leave the details to readers.
\end{proof}

Theorem~\ref{T:triangulation} answers the first half of \cite[Question~6.20]{LP16} for the case $q=3$ (i.e., the triangulation case), reconfirming a result by 
Haslegrave and Panagiotis in \cite{HP20}. To be precise, Haslegrave and Panagiotis answered the first statement of \cite[Question~6.20]{LP16} for $q=3$ and $q=4$, 
but the question is still open for general $q$. 
By the way, Theorem~\ref{T:triangulation} can  be used to obtain the exponential growth rate via \eqref{E:growth-iso}, which also confirms 
a result by Keller and Peyerimhoff in \cite{KP11}  for the case $q=3$.  Moreover, our method could give a better result for the exponential growth rate as follows. 

\begin{theorem}\label{P:prop}
Suppose $G=(V, E,F)$ is a triangulation such that $\deg v \geq 6$ for all $v \in V$, and let $B_n$  be the combinatorial balls centered at a fixed vertex $v_0$. If 
\begin{equation}\label{E:lav}
\frac{1}{|V(B_n )|} \sum_{v \in V(B_n)} \deg v  \geq c
\end{equation}
for sufficiently large $n$, where $c$ is a positive real number $> 6$, then we have 
\[
\mu(G) \geq \ln \frac{(c-4)+\sqrt{(c-2)(c-6)}}{2}.
\]
\end{theorem}
\begin{proof}
Let $S_n$ be the combinatorial sphere centered at $v_0$, and we observe that every vertex in $S_n$ has a neighbor in $S_{n+1}$. This is because 
$G$ is a non-positively curved triangulations, hence  the distance function on $V$ cannot have local maxima (see \cite{BP01} or \cite{BP06}).
Therefore  $(B_n)^+= B_{n+1}$ and $(B_{n+1})^- = B_n$ for every $n$. Moreover, $B_n$ is simply connected and $b B_n = b_i B_{n}$ is a simple cycle.  
Thus we see that each inequality in the proof of Lemma~\ref{ML1} becomes an equality in our case, and obtain
\[
|V(b B_{n+1})| = 6+ |b B_n| - 6 \sum_{v \in V(B_n)} \kappa(v),
\]
or
\begin{equation}\label{cp-ref}
s_{n+1}  = 6+ s_n +\sum_{v \in V(B_n)} (\deg v - 6) 
\end{equation}
for  $n =1,2, \ldots$, where we used the notation $s_n = |V(S_n)|$. (See Section~8 of \cite{Oh22} for a more detailed explanation of \eqref{cp-ref}.)
Therefore \eqref{E:lav} implies  
\[
s_{n+1} \geq 6 +s_n +(c-6) (s_0 + s_1 + \cdots + s_n)
\]
for sufficiently large $n$, say $n \geq n_0$. Let $N \geq n_0$, and define $\{ a_n \}_{n=n_0}^N$ as the sequence satisfying 
\[
\begin{cases}
a_{n+1} =  6+ a_{n} +   (c-6) (s_0 + \cdots + s_{n_0-1} + a_{n_0} + a_{n_0 +1} + \cdots + a_n)  \quad \mbox{ for } n \geq n_0,\\
a_{n_0} =t,
\end{cases}
\]
where $t$ is defined so that $a_N = s_N$. Then we have $\sum_{k=n_0}^{N-1} s_k \leq \sum_{k=n_0}^{N-1} a_k$ as before. 
Let $$\alpha = \frac{1}{2} \left\{ c-4 + \sqrt{(c-4)^2 -4} \right\}>1,$$ and observe that 
\[
a_{n_0 +1} - \alpha \cdot a_{n_0} = 6 + \left( \frac{1}{\alpha} -1 \right) a_{n_0} + \left( \alpha + \frac{1}{\alpha} -2 \right) (s_0 + \cdots + s_{n_0 -1})
\]
is bounded above by a constant not depending on $t= a_{n_0}$, hence not on the choice of $N$. But we know
\[
a_{n+2} - \alpha \cdot a_{n+1} =  \frac{1}{\alpha} \cdot ( a_{n+1} - \alpha \cdot a_n)
\]
for $n \geq n_0$, so $a_{N} - \alpha \cdot a_{N-1} \leq 6$ for sufficiently large $N$. Therefore
\[
\begin{aligned}
(c-6) & \cdot |V(B_{N-1})|  = (c-6) (s_0 + s_1 + \cdots + s_{N-1}) \\
& \leq (c-6) (s_0 + \cdots + s_{n_0-1} + a_{n_0} + a_{n_0 +1} + \cdots + a_{N-1}) \\
                         & = a_N - a_{N-1} -6 \leq a_N - \frac{a_N}{\alpha} + \frac{6}{\alpha} - 6 \leq \left( 1 - \frac{1}{\alpha} \right) \cdot |V(S_N)|.
\end{aligned}
\]
Because $\alpha + 1/\alpha = c-4$, this inequality implies $(\alpha -1 )\cdot |V(B_{N-1}) | \leq |V(S_N)|$ for sufficiently large $N$.
Now computation similar to \eqref{E:ball-sphere} and \eqref{E:growth-iso}  proves the theorem. We leave the details to readers.
\end{proof}

\section{Isoperimetric constants on locally tessellating plane graphs}\label{S:last}
As we mentioned in the introduction, Theorem~\ref{T1} is still valid even when the graph $G$ is a local tessellation of the plane; i.e., even when $G$ has some infinigons. 
In fact, if a finite graph $S \subset G$ is given, then it definitely cannot contain infinigons in its face set, hence the computation in Sections~\ref{S:CGB}--\ref{ProofT1} is still valid and we have
the same conclusions for $\imath^* (G)$ and $\imath_\sigma^* (G)$. For the isoperimetric constants $\imath(G)$ and $\imath_\sigma (G)$, there are several ways to overcome the
hurdle, and one of them is to construct a tessellation $G'$ that includes a subgraph $S'$ which is isomorphic to $S$. Note that such $G'$ always exists. But perhaps a better way is to compute
isoperimetric constants directly. First, we observe that the inequality \eqref{E2} essentially says that
\begin{equation}\label{E:final}
|V(b S)|+ q  + \frac{pq - 2p -2q}{2}\cdot |V(S)|\leq \frac{q-2}{2} \cdot |\partial S|
\end{equation}
for every \emph{induced} graph $S \subset G$, where $G$ satisfies the assumptions in Theorem~\ref{T1}.  Then by \eqref{inequal-2} we have 
\begin{equation}\label{E:facebound}
|V(S)| \leq  \frac{\alpha}{\alpha-1} \cdot |V(b S)|,
\end{equation}
hence \eqref{E:final} implies
\[
\frac{\alpha-1}{\alpha} \cdot |V(S)| + \frac{1}{2} \left( \alpha + \frac{1}{\alpha} -2 \right) |V(S)| \leq \frac{q-2}{2} \cdot |\partial S|,
\]
or
\[
\frac{|\partial S|}{|V(S)|}\geq  \frac{1}{q-2} \left( \alpha -\frac{1}{\alpha} \right) = (p-2) \sqrt{1 - \frac{4}{(p-2)(q-2)}}.
\]
This gives the lower bound for $\imath(G)$. The computation for the lower bound of $\imath_\sigma (G)$ is a little bit complicated. Note that for simply connected $S$ we have
$\sum_{v \in V(S)} \deg v = 2 |E(S)| + |\partial S|$. Thus  computation similar to \eqref{E:inequal-1} yields
\[
\frac{p-2}{2p} \sum_{v \in V(S)} \deg v  \leq \frac{1}{2} \cdot |\partial S| + |F(S)|.
\]
But \eqref{E:inequal-1} and \eqref{E:facebound} imply 
\[
\frac{q-2}{2} \cdot |F(S)| \leq \frac{1}{2} \left( 1 + \frac{1}{\alpha} \right)  |V(S)| \leq \frac{1}{2p} \left( 1 + \frac{1}{\alpha} \right)  \sum_{v \in V(S)} \deg v,
\]
and the last two inequalities give the lower bound of $\imath_\sigma (G)$.

The case $q = \infty$ is also allowed in Theorem~\ref{T2}. In this case there is no assumption about face degrees, and the constants $\imath^*(G)$ and $\imath^*_\sigma (G)$
look meaningless. For $\imath(G)$ and $\imath_\sigma(G)$, one can immediately prove a corresponding statement by comparing $G$ with the $p$-regular tree, and  
interpreting $\Phi(p, \infty)$ appropriately such as $\Phi(p, \infty) = p-2$.

\section*{Acknowledgement}
This research was supported by Basic Science Research Program through the National Research Foundation of Korea(NRF) funded by the Ministry of Education(NRF-2017R1D1A1B03034665).
 
\bibliographystyle{plain}

\bibliography{iso}

\begin{thebibliography}{10}

\bibitem{AH23}
Yohji Akama and Bobo Hua.
\newblock Hyperbolic polyhedral surfaces with regular faces.
\newblock {\em Discrete Math.}, 346(1):Paper No. 113213, 16, 2023.

\bibitem{AZ}
A.~D. Aleksandrov and V.~A. Zalgaller.
\newblock {\em Intrinsic geometry of surfaces}.
\newblock Translated from the Russian by J. M. Danskin. Translations of
  Mathematical Monographs, Vol. 15. American Mathematical Society, Providence,
  R.I., 1967.

\bibitem{Alon}
N.~Alon.
\newblock Eigenvalues and expanders.
\newblock volume~6, pages 83--96. 1986.
\newblock Theory of computing (Singer Island, Fla., 1984).

\bibitem{ABH18}
Omer Angel, Itai Benjamini, and Nizan Horesh.
\newblock An isoperimetric inequality for planar triangulations.
\newblock {\em Discrete Comput. Geom.}, 59(4):802--809, 2018.

\bibitem{BKW}
Frank Bauer, Matthias Keller, and Rados\l aw~K. Wojciechowski.
\newblock Cheeger inequalities for unbounded graph {L}aplacians.
\newblock {\em J. Eur. Math. Soc. (JEMS)}, 17(2):259--271, 2015.

\bibitem{BP01}
O.~Baues and N.~Peyerimhoff.
\newblock Curvature and geometry of tessellating plane graphs.
\newblock {\em Discrete Comput. Geom.}, 25(1):141--159, 2001.

\bibitem{BP06}
Oliver Baues and Norbert Peyerimhoff.
\newblock Geodesics in non-positively curved plane tessellations.
\newblock {\em Adv. Geom.}, 6(2):243--263, 2006.

\bibitem{BS}
I.~Benjamini and O.~Schramm.
\newblock Every graph with a positive {C}heeger constant contains a tree with a
  positive {C}heeger constant.
\newblock {\em Geom. Funct. Anal.}, 7(3):403--419, 1997.

\bibitem{BMS}
N.~L. Biggs, Bojan Mohar, and John Shawe-Taylor.
\newblock The spectral radius of infinite graphs.
\newblock {\em Bull. London Math. Soc.}, 20(2):116--120, 1988.

\bibitem{Bol}
B\'{e}la Bollob\'{a}s.
\newblock The isoperimetric number of random regular graphs.
\newblock {\em European J. Combin.}, 9(3):241--244, 1988.

\bibitem{BoMur}
J.~A. Bondy and U.~S.~R. Murty.
\newblock {\em Graph theory}, volume 244 of {\em Graduate Texts in
  Mathematics}.
\newblock Springer, New York, 2008.

\bibitem{BZ}
Yu.~D. Burago and V.~A. Zalgaller.
\newblock {\em Geometric inequalities}, volume 285 of {\em Grundlehren der
  Mathematischen Wissenschaften [Fundamental Principles of Mathematical
  Sciences]}.
\newblock Springer-Verlag, Berlin, 1988.
\newblock Translated from the Russian by A. B. Sosinski\u{\i}, Springer Series
  in Soviet Mathematics.

\bibitem{Che}
Jeff Cheeger.
\newblock A lower bound for the smallest eigenvalue of the {L}aplacian.
\newblock In {\em Problems in analysis ({P}apers dedicated to {S}alomon
  {B}ochner, 1969)}, pages 195--199. 1970.

\bibitem{Chen09}
Beifang Chen.
\newblock The {G}auss-{B}onnet formula of polytopal manifolds and the
  characterization of embedded graphs with nonnegative curvature.
\newblock {\em Proc. Amer. Math. Soc.}, 137(5):1601--1611, 2009.

\bibitem{CC08}
Beifang Chen and Guantao Chen.
\newblock Gauss-{B}onnet formula, finiteness condition, and characterizations
  of graphs embedded in surfaces.
\newblock {\em Graphs Combin.}, 24(3):159--183, 2008.

\bibitem{CDP}
M.~Coornaert, T.~Delzant, and A.~Papadopoulos.
\newblock {\em G\'{e}om\'{e}trie et th\'{e}orie des groupes}, volume 1441 of
  {\em Lecture Notes in Mathematics}.
\newblock Springer-Verlag, Berlin, 1990.
\newblock Les groupes hyperboliques de Gromov. [Gromov hyperbolic groups], With
  an English summary.

\bibitem{HP}
Pierre de~la Harpe.
\newblock {\em Topics in geometric group theory}.
\newblock Chicago Lectures in Mathematics. University of Chicago Press,
  Chicago, IL, 2000.

\bibitem{DeMo07}
Matt DeVos and Bojan Mohar.
\newblock An analogue of the {D}escartes-{E}uler formula for infinite graphs
  and {H}iguchi's conjecture.
\newblock {\em Trans. Amer. Math. Soc.}, 359(7):3287--3300, 2007.

\bibitem{Dies}
Reinhard Diestel.
\newblock {\em Graph theory}, volume 173 of {\em Graduate Texts in
  Mathematics}.
\newblock Springer-Verlag, Berlin, third edition, 2005.

\bibitem{DK}
J.~Dodziuk and W.~S. Kendall.
\newblock Combinatorial {L}aplacians and isoperimetric inequality.
\newblock In {\em From local times to global geometry, control and physics
  ({C}oventry, 1984/85)}, volume 150 of {\em Pitman Res. Notes Math. Ser.},
  pages 68--74. Longman Sci. Tech., Harlow, 1986.

\bibitem{Doz}
Jozef Dodziuk.
\newblock Difference equations, isoperimetric inequality and transience of
  certain random walks.
\newblock {\em Trans. Amer. Math. Soc.}, 284(2):787--794, 1984.

\bibitem{DKarp}
Jozef Dodziuk and Leon Karp.
\newblock Spectral and function theory for combinatorial {L}aplacians.
\newblock In {\em Geometry of random motion ({I}thaca, {N}.{Y}., 1987)},
  volume~73 of {\em Contemp. Math.}, pages 25--40. Amer. Math. Soc.,
  Providence, RI, 1988.

\bibitem{Fed82}
Pasquale~Joseph Federico.
\newblock {\em Descartes on polyhedra}, volume~4 of {\em Sources in the History
  of Mathematics and Physical Sciences}.
\newblock Springer-Verlag, New York-Berlin, 1982.
\newblock A study of the {{\i}t De solidorum elementis}.

\bibitem{Fuj}
Koji Fujiwara.
\newblock The {L}aplacian on rapidly branching trees.
\newblock {\em Duke Math. J.}, 83(1):191--202, 1996.

\bibitem{Gel87}
Peter Gerl.
\newblock Amenable groups and amenable graphs.
\newblock In {\em Harmonic analysis ({L}uxembourg, 1987)}, volume 1359 of {\em
  Lecture Notes in Math.}, pages 181--190. Springer, Berlin, 1988.

\bibitem{Gel}
Peter Gerl.
\newblock Random walks on graphs with a strong isoperimetric property.
\newblock {\em J. Theoret. Probab.}, 1(2):171--187, 1988.

\bibitem{Ghi23}
Luca Ghidelli.
\newblock On the largest planar graphs with everywhere positive combinatorial
  curvature.
\newblock {\em J. Combin. Theory Ser. B}, 158(part 2):226--263, 2023.

\bibitem{GhyHar}
\'{E}tienne Ghys and Pierre de~la Harpe.
\newblock Le bord d'un espace hyperbolique.
\newblock In {\em Sur les groupes hyperboliques d'apr{\`e}s {M}ikhael {G}romov
  ({B}ern, 1988)}, volume~83 of {\em Progr. Math.}, pages 117--134.
  Birkh\"{a}user Boston, Boston, MA, 1990.

\bibitem{Gro}
M.~Gromov.
\newblock Hyperbolic groups.
\newblock In {\em Essays in group theory}, volume~8 of {\em Math. Sci. Res.
  Inst. Publ.}, pages 75--263. Springer, New York, 1987.

\bibitem{HJL}
Olle H\"aggstr\"om, Johan Jonasson, and Russell Lyons.
\newblock Explicit isoperimetric constants and phase transitions in the
  random-cluster model.
\newblock {\em Ann. Probab.}, 30(1):443--473, 2002.

\bibitem{HP20}
John Haslegrave and Christoforos Panagiotis.
\newblock Site percolation and isoperimetric inequalities for plane graphs.
\newblock {\em Random Structures \& Algorithms}, n/a(n/a).

\bibitem{Hig}
Yusuke Higuchi.
\newblock Combinatorial curvature for planar graphs.
\newblock {\em J. Graph Theory}, 38(4):220--229, 2001.

\bibitem{HS}
Yusuke Higuchi and Tomoyuki Shirai.
\newblock Isoperimetric constants of {$(d,f)$}-regular planar graphs.
\newblock {\em Interdiscip. Inform. Sci.}, 9(2):221--228, 2003.

\bibitem{HJ15}
Bobo Hua, J\"{u}rgen Jost, and Shiping Liu.
\newblock Geometric analysis aspects of infinite semiplanar graphs with
  nonnegative curvature.
\newblock {\em J. Reine Angew. Math.}, 700:1--36, 2015.

\bibitem{HuS19}
Bobo Hua and Yanhui Su.
\newblock The set of vertices with positive curvature in a planar graph with
  nonnegative curvature.
\newblock {\em Adv. Math.}, 343:789--820, 2019.

\bibitem{HS20}
Bobo Hua and Yanhui Su.
\newblock The first gap for total curvatures of planar graphs with nonnegative
  curvature.
\newblock {\em J. Graph Theory}, 93(3):395--439, 2020.

\bibitem{HS22}
Bobo Hua and Yanhui Su.
\newblock Total curvature of planar graphs with nonnegative combinatorial
  curvature.
\newblock {\em Trans. Amer. Math. Soc.}, 375(12):8423--8444, 2022.

\bibitem{Kel10}
Matthias Keller.
\newblock The essential spectrum of the {L}aplacian on rapidly branching
  tessellations.
\newblock {\em Math. Ann.}, 346(1):51--66, 2010.

\bibitem{Kel11}
Matthias Keller.
\newblock Curvature, geometry and spectral properties of planar graphs.
\newblock {\em Discrete Comput. Geom.}, 46(3):500--525, 2011.

\bibitem{Kel17}
Matthias Keller.
\newblock Geometric and spectral consequences of curvature bounds on
  tessellations.
\newblock In {\em Modern approaches to discrete curvature}, volume 2184 of {\em
  Lecture Notes in Math.}, pages 175--209. Springer, Cham, 2017.

\bibitem{KM}
Matthias Keller and Delio Mugnolo.
\newblock General {C}heeger inequalities for {$p$}-{L}aplacians on graphs.
\newblock {\em Nonlinear Anal.}, 147:80--95, 2016.

\bibitem{KP11}
Matthias Keller and Norbert Peyerimhoff.
\newblock Cheeger constants, growth and spectrum of locally tessellating planar
  graphs.
\newblock {\em Math. Z.}, 268(3-4):871--886, 2011.

\bibitem{KPP}
Matthias Keller, Norbert Peyerimhoff, and Felix Pogorzelski.
\newblock Sectional curvature of polygonal complexes with planar substructures.
\newblock {\em Adv. Math.}, 307:1070--1107, 2017.

\bibitem{KN}
Aleksey Kostenko and Noema Nicolussi.
\newblock Spectral estimates for infinite quantum graphs.
\newblock {\em Calc. Var. Partial Differential Equations}, 58(1):Paper No. 15,
  40, 2019.

\bibitem{LPZ}
Serge Lawrencenko, Michael~D. Plummer, and Xiaoya Zha.
\newblock Isoperimetric constants of infinite plane graphs.
\newblock {\em Discrete Comput. Geom.}, 28(3):313--330, 2002.

\bibitem{LevPer}
David~A. Levin and Yuval Peres.
\newblock {\em Markov chains and mixing times}.
\newblock American Mathematical Society, Providence, RI, 2017.
\newblock Second edition of [ MR2466937], With contributions by Elizabeth L.
  Wilmer, With a chapter on ``Coupling from the past'' by James G. Propp and
  David B. Wilson.

\bibitem{LP16}
Russell Lyons and Yuval Peres.
\newblock {\em Probability on trees and networks}, volume~42 of {\em Cambridge
  Series in Statistical and Probabilistic Mathematics}.
\newblock Cambridge University Press, New York, 2016.

\bibitem{Moh88}
Bojan Mohar.
\newblock Isoperimetric inequalities, growth, and the spectrum of graphs.
\newblock {\em Linear Algebra Appl.}, 103:119--131, 1988.

\bibitem{Moh89}
Bojan Mohar.
\newblock Isoperimetric numbers of graphs.
\newblock {\em J. Combin. Theory Ser. B}, 47(3):274--291, 1989.

\bibitem{Moh91}
Bojan Mohar.
\newblock Some relations between analytic and geometric properties of infinite
  graphs.
\newblock volume~95, pages 193--219. 1991.
\newblock Directions in infinite graph theory and combinatorics (Cambridge,
  1989).

\bibitem{Moh92}
Bojan Mohar.
\newblock Isoperimetric numbers and spectral radius of some infinite planar
  graphs.
\newblock {\em Math. Slovaca}, 42(4):411--425, 1992.

\bibitem{Moh02}
Bojan Mohar.
\newblock Light structures in infinite planar graphs without the strong
  isoperimetric property.
\newblock {\em Trans. Amer. Math. Soc.}, 354(8):3059--3074, 2002.

\bibitem{Nev}
Rolf Nevanlinna.
\newblock {\em Analytic functions}.
\newblock Translated from the second German edition by Phillip Emig. Die
  Grundlehren der mathematischen Wissenschaften, Band 162. Springer-Verlag, New
  York-Berlin, 1970.

\bibitem{Nic}
Noema Nicolussi.
\newblock Strong isoperimetric inequality for tessellating quantum graphs.
\newblock In {\em Discrete and Continuous Models in the Theory of Networks},
  pages 271--290. Springer International Publishing, 2020.

\bibitem{Oh14}
Byung-Geun Oh.
\newblock Duality properties of strong isoperimetric inequalities on a planar
  graph and combinatorial curvatures.
\newblock {\em Discrete Comput. Geom.}, 51(4):859--884, 2014.

\bibitem{Oh15-2}
Byung-Geun Oh.
\newblock Hyperbolic notions on a planar graph of bounded face degree.
\newblock {\em Bull. Korean Math. Soc.}, 52(4):1305--1319, 2015.

\bibitem{Oh17}
Byung-Geun Oh.
\newblock On the number of vertices of positively curved planar graphs.
\newblock {\em Discrete Math.}, 340(6):1300--1310, 2017.

\bibitem{Oh22}
Byung-Geun Oh.
\newblock Some criteria for circle packing types and combinatorial
  {G}auss-{B}onnet theorem.
\newblock {\em Trans. Amer. Math. Soc.}, 375(2):753--797, 2022.

\bibitem{OS16}
Byung-Geun Oh and Jeehyeon Seo.
\newblock Strong isoperimetric inequalities and combinatorial curvatures on
  multiply connected planar graphs.
\newblock {\em Discrete Comput. Geom.}, 56(3):558--591, 2016.

\bibitem{Old17}
Paul~Richard Oldridge.
\newblock {\em Characterizing the polyhedral graphs with positive combinatorial
  curvature}.
\newblock PhD thesis, 2017.

\bibitem{PRT}
Ori Parzanchevski, Ron Rosenthal, and Ran~J. Tessler.
\newblock Isoperimetric inequalities in simplicial complexes.
\newblock {\em Combinatorica}, 36(2):195--227, 2016.

\bibitem{Res}
Yu.~G. Reshetnyak.
\newblock Two-dimensional manifolds of bounded curvature.
\newblock In {\em Geometry, {IV}}, volume~70 of {\em Encyclopaedia Math. Sci.},
  pages 3--163, 245--250. Springer, Berlin, 1993.

\bibitem{Soa90}
P.~M. Soardi.
\newblock Recurrence and transience of the edge graph of a tiling of the
  {E}uclidean plane.
\newblock {\em Math. Ann.}, 287(4):613--626, 1990.

\bibitem{Soa}
Paolo~M. Soardi.
\newblock {\em Potential theory on infinite networks}, volume 1590 of {\em
  Lecture Notes in Mathematics}.
\newblock Springer-Verlag, Berlin, 1994.

\bibitem{Sto76}
David~A. Stone.
\newblock A combinatorial analogue of a theorem of {M}yers.
\newblock {\em Illinois J. Math.}, 20(1):12--21, 1976.

\bibitem{SY}
Liang Sun and Xingxing Yu.
\newblock Positively curved cubic plane graphs are finite.
\newblock {\em J. Graph Theory}, 47(4):241--274, 2004.

\bibitem{Woe98}
Wolfgang Woess.
\newblock A note on tilings and strong isoperimetric inequality.
\newblock {\em Math. Proc. Cambridge Philos. Soc.}, 124(3):385--393, 1998.

\bibitem{Zuk}
Andrzej \.{Z}uk.
\newblock On the norms of the random walks on planar graphs.
\newblock {\em Ann. Inst. Fourier (Grenoble)}, 47(5):1463--1490, 1997.

\end{thebibliography}

\end{document}